\pdfoutput=1

\documentclass[a4paper,final,fleqn]{article}

\usepackage[lm,grey]{janm}
\usepackage{math}
\usepackage[normalem]{ulem}

\usepackage[numbers]{natbib}
\renewcommand{\thesection}{\arabic{section}}

\def\@biblabel#1{#1.}

\DeclareMathAlphabet{\mathpzc}{OT1}{pzc}{m}{it}

\usepackage{showkeys}
\usepackage[textwidth=3.5cm,color=lightgray]{todonotes}

\usepackage[inner=4.5cm,outer=4.5cm,marginparwidth=3cm, marginparsep=.5cm]{geometry}

\newcommand{\asin}{\sin^{-1}}

\newcommand{\ld}{\text{\tiny\ensuremath{\bullet}}}
\newcommand{\dDl}{\Dl^{\ld}}

\newcommand{\Hm}{H}

\newcommand{\Wp}{\Ws}
\renewcommand{\vs}{w}

\makeatletter
\newcommand\newsubcommand[3]{\def#1{#2\sc@sub{#3}}}
\def\sc@sub#1{\def\sc@thesub{#1}\@ifnextchar_{\sc@mergesubs}{_{\sc@thesub}}}
\def\sc@mergesubs_#1{_{\sc@thesub#1}}
\makeatother

\newcommand{\FL}{\Fs\ell}

\newcommand{\oa}{\om}
\newcommand{\op}{\om}

\title{\boldmath On the wellposedness of the defocusing mKdV equation below $L^{2}$}
\author{Thomas Kappeler\footnote{Partially supported by the Swiss National Science Foundation
}, Jan-Cornelius Molnar\footnote{Partially supported by the Swiss National Science Foundation
}}
\date{\today}

\newcommand{\rmKdV}{mKdV${}_{\#}$\xspace}
\newcommand{\fmKdV}{mKdV${}^{-}$\xspace}
\newcommand{\rfmKdV}{mKdV${}_{\#}^{-}$\xspace}
\newcommand{\nbr}[1]{\{#1\}}

\begin{document}

\maketitle

\begin{abstract}
We prove that the renormalized defocusing mKdV equation on the circle is locally in time $C^{0}$-wellposed on the Fourier Lebesgue space $\FL^p$
for any $2 < p < \infty$. The result implies that the defocusing mKdV equation itself is illposed on these spaces
since the renormalizing phase factor becomes infinite.
The proof is based on the fact that the mKdV equation is an integrable PDE whose Hamiltonian is in the NLS hierarchy.
A key ingredient is a novel way of representing the bi-infinite sequence of frequencies of the renormalized defocusing mKdV equation,
allowing to analytically extend them to $\FL^p$ for any $2 \le p < \infty$ and to deduce asymptotics for $n \to \pm \infty$.


\end{abstract}


\section{Introduction}

Consider the defocusing modified Korteweg-de Vries (mKdV) equation on the circle $\T = \R/\Z$
\begin{equation}
  \label{mkdv}
  \partial_{t}u = -\partial_{x}^{3}u + 6u^{2}\partial_{x}u
\end{equation}
and its renormalized version
\begin{equation}
  \label{rmkdv}
  \partial_{t}u = -\partial_{x}^{3}u + 6\p*{u^{2}-\int_{0}^{1} u^{2}\,\dx}\partial_{x}u
\end{equation}
referred to as the \rmKdV equation.

According to~\cite{Kappeler:2005gt}, for any initial datum $u\in H^{s} \equiv H^{s}(\T,\R)$ with $s\ge 0$, there exists a unique, global in time solution $u(t,x) = u(t,x,q)$ of~\eqref{mkdv}, $u\in C(\R,H^{s})$. In particular, for any time $t\in\R$ and $T > 0$, the nonlinear evolution operator
\[
  \Sc^{t} = \Sc(t,\cdot)\colon H^{s}\to H^{s},\qquad v\mapsto u(t,\cdot,v),
\]
and the uniquely defined solution map
\[
  \Sc\colon H^{s}\to C([-T,T],H^{s}),\qquad v\mapsto u(\cdot,\cdot,v),
\]
are well defined and continuous.

To investigate the wellposedness of mKdV and \rmKdV below $L^{2}$, we introduce for any $1\le p \le \infty$ the \emph{Fourier Lebesgue spaces} $\FL^{p}$ consisting of 1-periodic distributions whose Fourier coefficients are in $\ell^{p}$
\[
  \FL^{p} = \setdef{v\in S'(\T,\R)}{(\hat{v}_{n})\in\ell^{p}},\qquad \n{v}_{\FL^{p}} \defl \n{(\hat{v})}_{\ell^{p}},
\]
where $\hat{v}_{n}$ denotes the $n$th Fourier coefficient of $v$ in $S'(\T,\R)$.

For any $1\le p \le \infty$, the space $\FL^{p}$ is a real Banach space. Moreover, for any $2 < p < \infty$, we have the following continuous embeddings
\[
  \FL^{1} \opento C^{0} \opento L^{2} = \FL^{2} \opento L^{p'} 
  \opento
  \left\{
  \begin{aligned}
  L^{1} \opento \Ms\;\\
  L^{p'} \opento \FL^{p}
  \end{aligned}
  \right\}
  \opento \FL^{\infty},
\]
where $p'$ denotes the Lebesgue exponent conjugated to $p$ and $\Ms$ denotes the space of finite Borel measures on $\T$. The space $\FL^{1}$ is called the \emph{Wiener algebra} and $\FL^{\infty}$ is the space of \emph{pseudo measures}.
 We point out that for $2 < p < \infty$, the space $\FL^{p}$ is much larger than $L^{p'}$. In particular, it contains elements which are not measures but more singular distributions.

A continuous curve $\gm\colon (a,b)\to \FL^{p}$, $\gm(0) = v$, is called a \emph{solution} of an equation such as the mKdV equation in $\FL^{p}$ with initial datum $v$ if and only if for any sequence of $C^{\infty}$-potentials $(v_{k})_{k\ge 1}$ converging to $v$ in $\FL^{p}$, the corresponding sequence $(\Sc(t,v_{k}))_{k\ge 1}$ of solutions of~\eqref{mkdv} with initial data $v_{k}$ converges to $\gm(t)$ in $\FL^{p}$ for any $t\in (a,b)$.
mKdV is said to be \emph{locally in time $C^{0}$-wellposed} if for any initial datum $w\in \FL^{p}$ there exists a neighborhood $U$ and a time $T > 0$ so that the initial value problem~\eqref{mkdv} for any initial value $v\in U$ admits a solution $\Sc(\cdot,v)$ in the aforementioned sense which is defined on the time interval $[-T,T]$, and the solution map $\Sc\colon U\to C^{0}([-T,T],\FL^{p})$ is continuous.
mKdV is said to be \emph{globally in time $C^{0}$-wellposed on an open subset $U\subset \FL^{p}$} if for any initial datum $v\in U$ the initial value problem~\eqref{mkdv} admits a solution $\Sc(\cdot,v)$ in the aforementioned sense which is globally defined in time and the solution map $\Sc\colon U\to C([-T,T],\FL^{p})$ is continuous for every $T > 0$.
mKdV is said to be \emph{globally in time (uniformly/$C^{k}$/$C^{\om}$) wellposed on $U\subset\FL^{p}$} if the solution map $\Sc\colon U\to C([-T,T],\FL^{p})$ is (uniformly continuous/$C^{k}$/$C^{\om}$) for every $T > 0$.

\begin{thm}
\label{thm:wp-defocusing-mkdv}
The \rmKdV equation~\eqref{rmkdv} is locally in time $C^{0}$-wellposed in $\FL^{p}$ for any $2 < p < \infty$ and it is globally in time $C^{0}$-wellposed in $\FL^{p}$ in a  neighborhood of $0$. As a consequence, for any $2 < p < \infty$, the solution map $\Sc$ of the mKdV equation does not extend continuously to initial data in $\FL^{p}\setminus L^{2}$.\fish
\end{thm}

Our method also allows to prove partial corresponding results for the focusing modified Korteweg-de Vries (\fmKdV) equation on $\T = \R/\Z$
\[
  \partial_{t}u = -\partial_{x}^{3}u - 6u^{2}\partial_{x}u
\]
and its renormalized version
\[
  \partial_{t}u = -\partial_{x}^{3}u - 6\p*{u^{2}-\int_{0}^{1}u^{2}\,\dx}\partial_{x}u
\]
referred to as \rfmKdV equation. More precisely, for any $2 < p < \infty$, the \rfmKdV equation is globally in time $C^{0}$-wellposed in $\FL^{p}$ for initial data in a neighborhood of $0$, whereas the solution map of the \fmKdV equation does not extend continuously to small initial data in $\FL^{p}\setminus L^{2}$ for any $2 < p < \infty$.

Finally, using the Miura map one can deduce from corresponding results in~\cite{Kappeler:2016uj} for the KdV equation that for any $0 < s < 1/2$, the mKdV equation is nowhere locally uniformly $C^{0}$-wellposed on the submanifold
\[
  \setdef*{v\in H^{s}(\T)}{\int_{\T} v^{2}\,\dx = d}
\] 
for any $d > 0$.

\paragraph{Related results}
The wellposedness of the mKdV equation on $\T$ has been extensively studied - cf. e.g.~\cite[Introduction~1]{Molinet:2012il} for an account on the many results obtained so far. In particular, based on the seminal work of \citet{Bourgain:1993cl}, it was shown in~\cite{Colliander:2003fv} that the mKdV equation is globally uniformly $C^{0}$-wellposed on the submanifolds
\[
  M_{d}^{s} = \setdef*{v\in H^{s}(\T)}{\int_{\T} v^{2}\,\dx = d}
\]
with $s\ge 1/2$ and $d > 0$ arbitrary. In~\cite{Kappeler:2005gt} it was proved that the mKdV equation is globally $C^{0}$-wellposed in $H^{s}(\T)$ for any $s\ge 0$. The solutions constructed in~\cite{Kappeler:2005gt} are shown to be weak (in an appropriate sense) by~\cite{Molinet:2012il} (cf. also ~\cite{Nakanishi:2010ix,Takaoka:2004bg} for related results).
See also 
\cite{Grunrock:2004vw,Nguyen:2008vu} for results on the local wellposedness of the mKdV equation in Fourier Lebesgue spaces of higher regularity. On the other hand, it was pointed out in~\cite[Section~6]{Bourgain:1997gg} that the mKdV equation is not $C^{3}$-wellposed in $H^{s}(\T)$ for $s < 1/2$, whereas in~\cite{Christ:2003tx} it is proved that it is not uniformly $C^{0}$-wellposed in $H^{s}(\T)$ for $-1 < s < 1/2$. More recently, \citet{Molinet:2012il} proved that the mKdV equation is ill-posed on $H^{s}(\T)$ for any $s < 0$ in the sense that for any $T > 0$, the solution map $(C^{\infty}(\T),\n{\cdot}_{s}) \to \Dc'((0,T)\times \T)$ is discontinuous at any nonconstant initial datum $v\in H^{\infty}(\T)$. Here, $(C^{\infty}(\T),\n{\cdot}_{s})$ denotes the vector space $C^{\infty}(\T)$ endowed with the Sobolev norm $\n{\cdot}_{s}$.

\paragraph{Method of proof}

The mKdV equation is closely related to the \emph{\nls system}
\begin{equation}
\label{nls-sys}
\begin{split}
  \ii \partial_{t}\phm &= \phantom{-}\partial_{\php}H = -\partial_{xx}\phm + 2\php\phm^{2},\\
  \ii \partial_{t}\php &= -\partial_{\phm}H = \phantom{-}\partial_{xx}\php - 2\phm\php^{2}.
\end{split}
\end{equation}
This system can be viewed as a Hamiltonian PDE with Hamiltonian
\begin{equation}
  \label{nls-ham}
  H_{NLS}(\ph) = \int_{\T} (\partial_{x}\php\partial_{x}\phm + \php^{2}\phm^{2}) \,\dx
\end{equation}
on the phase space $H^{s}_{c} \defl H^{s}(\T,\C)\times H^{s}(\T,\C)$, $s\ge 0$, with Poisson bracket
\begin{equation}
  \label{nls-poi}
  \nbr{F,G} 
  \defl 
  -\ii\int_{\T} 
  (\partial_{\phm} F\, \partial_{\php} G 
    - \partial_{\php}F\,\partial_{\phm} G)\,\dx.
\end{equation}
Here $\phm$, $\php$ denote the two components of $\ph=(\phm,\php)\in H^{s}_{c}$, and $\partial_{\phm} F$, $\partial_{\php} F$ denote the two components of the $L^2$-gradient $\partial F=(\partial_{\phm}F,\partial_{\php}F)$ of a $C^1$-functional $F$ on $H_{c}^{s}$.

The Hamiltonian $\Hm_{NLS}(\ph)$ admits for any $s\ge 0$ the invariant real subspaces
\[
  H_{r}^{s} \defl \setdef{\ph\in H_{c}^{s}}{\php =  \ob{\phm}},\quad
  H_{i}^{s} \defl \setdef{\ph\in H_{c}^{s}}{\php = -\ob{\phm}}.
\]
Elements of $H_{r}^{s}$ are called potentials of \emph{real type}.

When \eqref{nls-sys} is restricted to $H_{r}^{s}$, with $\ph = (v,\ob{v})$, one obtains the \emph{defocusing NLS} (dNLS) equation
\[
  \ii \partial_{t}v = \ii\nbr{v,\Hm_{NLS}} = -\partial_{xx}v + 2\abs{v}^{2}v ,\qquad \Hm_{NLS}(v,\ob{v}) = \int_{\T} (\abs{v_{x}}^{2} + \abs{v}^{4}) \,\dx.
\]
Similarly, when \eqref{nls-sys} is restricted to $H_{i}^{m}$, with $\ph = (\ii v,\ii \ob{v})$, one obtains the \emph{focusing NLS} equation
\[
  \ii \partial_{t}v = \ii\nbr{v,\Hm_{NLS}} = -\partial_{xx}v - 2\abs{v}^{2}v ,\qquad
  \Hm_{NLS}(\ii v,\ii \ob{v}) = -\int_{\T} (\abs{v_{x}}^{2} - \abs{v}^{4}) \,\dx.
\]

The \nls system \eqref{nls-sys} admits an infinite sequence of recursively defined pairwise Poisson commuting integrals referred to as \emph{\nls hierarchy},
\begin{align*} 
  \Hm_{1}(\ph) &= \;\;\phantom{\ii}\int_{\T} \phm\php  \,\dx,\\
  \Hm_{2}(\ph) &= \frac{\ii}{2} \int_{\T} (\php\partial_{x}\phm - \phm\partial_{x}\php)\,\dx,\\
   \Hm_{NLS}(\ph) = \Hm_{3}(\ph)
    &= \;\;\phantom{\ii}\int_{\T} (\partial_{x}\phm\partial_{x}\php + \phm^{2}\php^{2}) \,\dx,\\
  \Hm_{4}(\ph) &= \;\;\ii\int_{\T} (\phm\partial_{xxx}\php - 3\phm^{2}\php\partial_{x}\php) \,\dx,
  \qquad
 \ldots
\end{align*}
In comparison with \cite{Grebert:2014iq}, the $n$th Hamiltonian of the \nls hierarchy (for $n\ge 2$) is multiplied by $(-\ii)^{n+1}$ to make the corresponding Hamiltonian flow real-valued for real-valued $\ph$.

The Hamiltonian $\Hm_{4}$ gives rise to the \emph{mKdV system}
\begin{equation}
\label{mkdv-sys}
\begin{split}
  \partial_{t} \phm &= \nbr{\phm,\Hm_{4}}
  					 = -\ii\partial_{\php}\Hm_{4}
					 = -\partial_{xxx}\phm + 6\phm\php\partial_{x}\phm,\\
  \partial_{t} \php &= \nbr{\php,\Hm_{4}}
  					 = \phantom{-}\ii\partial_{\phm}\Hm_{4}
					 = -\partial_{xxx}\php + 6\php\phm\partial_{x}\php.
\end{split}
\end{equation}
This system admits, for any $s\ge 0$, the real invariant subspaces
\[
  \Ec_{r}^{s} \defl \setdef{\ph \in H_{r}^{s}}{\php = \phm},\qquad
  \Ec_{i}^{s} \defl \setdef{\ph \in H_{i}^{s}}{\php = \phm}.
\]
When \eqref{mkdv-sys} is restricted to $\Ec_{r}^{s}$, with $\ph = (u,u)$ and $u$ real-valued, one obtains the defocusing mKdV equation
\[
  \partial_{t}u
   = -\ii \partial_{\php}H_{4}\big|_{(u,u)}
   = -\partial_{xxx}u + 6u^{2}\partial_{x} u.
\]
Similarly, when \eqref{mkdv-sys} is restricted to $\Ec_{i}^{s}$, with $\ph = (\ii u,\ii u)$ and $u$ real-valued, one obtains the focusing mKdV equation
\[
  \ii \partial_{t}u
   =  -\ii\partial_{\php}H_{4}\big|_{(\ii u,\ii u)}
   = \ii(-\partial_{xxx}u - 6u^{2}\partial_{x} u)
\]
or $\partial_{t}u = -\partial_{xxx}u - 6u^{2}\partial_{x} u$.

The NLS hierarchy is completely integrable in the strongest possible sense meaning that by~\cite{Grebert:2014iq} it admits global Birkhoff coordinates on $H_{r}^{s}$ for $s\ge 0$. They can be extended to the Fourier Lebesgue spaces. To give a precise statement, we introduce for any $2\le p < \infty$ the model space
\[
  \ell_{r}^{p} \defl \setdef{(z_{-},z_{+})\in \ell_{c}^{p}}{z_{+}=\ob{z_{-}}},\qquad
  \ell_{c}^{p} \defl \ell^{p}(\Z,\C)\times \ell^{p}(\Z,\C),
\]
and the phase space
\[
  \FL_{r}^{p} \defl \setdef{\ph = (\phm,\php)\in \FL_{c}^{p}}{\php=\ob{\phm}},\qquad
  \FL_{c}^{p} \defl \FL^{p}(\T,\C)\times \FL^{p}(\T,\C).
\]

\begin{thm}
\label{bhf}
For any $2 \le p < \infty$, there exists a complex neighborhood $\Ws^{p}\subset\FL_{c}^{p}$ of $\FL_{r}^{p}$ and an analytic map $\Ph_{p}\colon \Ws^{p}\to \ell_{c}^{p}$ with $\Ph_{p}(0) = 0$ so that the following holds:

\begin{equivenum}
\item
$\Phi_{p}\colon \Ws^{p}\to \ell_{c}^{p}$ is real analytic,

\item
for any $2\le p,q < \infty$, the maps $\Phi_{p}$ and $\Phi_{q}$ coincide on $\Ws^{p}\cap \Ws^{q}$,

\item
$\Phi_{p}$ is canonical,

\item the map $\Phi_{p}\colon \FL_{r}^{p}\to \ell_{r}^{p}$ is one-to-one, a local diffeomorphism at every point of $\FL_{r}^{p}$, and the image of this map is open and dense in $\ell_{r}^{p}$.

\item If $p=2$, then the map $\Phi_{p}$ coincides with the one constructed in~\cite{Grebert:2014iq}. In particular, it is onto and hence a bi-real analytic diffeomorphism.~\fish

\end{equivenum}
\end{thm}

For $1\le q\le \infty$ and $s\in\R$, let $\ell_{+}^{s,q}$ denote the positive quadrant of $\ell^{s,q}(\Z,\R)$ given by
\begin{equation}
  \label{ellp-pos-quad}
  \ell_{+}^{s,q} = \setdef{I=(I_{n})_{n\in\Z}\in \ell^{s,q}(\Z,\R)}{I_{n}\ge 0}.
\end{equation}
The Hamiltonian $\Hm_{4}$ is a real analytic function of $\ph$ on $H_{r}^{3/2}$ and hence a real analytic function of the actions on $\ell_{+}^{3,1}$ -- see \cite{Grebert:2014iq}. The corresponding frequencies
\[
  \om_{n}^{(4)} \defl \partial_{I_{n}}\Hm_{4},\qquad n\in\Z,
\]
thus are real analytic functions of the actions on $\ell_{+}^{3,1}$.

The equations of motion of the system~\eqref{mkdv-sys}, when expressed in Birkhoff coordinates on $\ell_{r}^{s,2}$ with $s \ge 3/2$, are
\[
  \partial_{t}z_{n}^{-} = -\ii \om_{n}^{(4)}z_{n}^{-},\quad
  \partial_{t}z_{n}^{+} = \ii \om_{n}^{(4)}z_{n}^{+},\quad
  n\in\Z.
\]

On $\ell_{+}^{3,1}$ the frequencies $\om_{n}^{(4)}$ have an asymptotic expansion for $\abs{n}\to \infty$ of the form
\begin{align}
  \label{exp-om-kdv1}
  \om_{n}^{(4)} &= (2n\pi)^{3} + 6\Hm_{2} + 12n\pi \Hm_{1} + O(n^{-1})
\end{align}
where according to~\cite{Grebert:2002wb,Grebert:2014iq}
\[
  \Hm_{1} = \sum_{n\in\Z} I_{n},\qquad \Hm_{2} = \sum_{n\in\Z} (2n\pi)I_{n}.
\]
In order to state our results on the analytic extensions of $\om_{n}^{(4)}$, $n\in\Z$, we need to normalize the frequencies as follows
\begin{equation}
  \label{omn-4-star}
  \om_{n}^{(4)\star} = \om_{n}^{(4)} - (2n\pi)^{3} - 6\Hm_{2} - 12n\pi\Hm_{1}.
\end{equation}

By a slight abuse of notation, in the sequel, we will often view the frequencies as functions of the potential. From the considerations above it follows that $\oa^{(4)\star} = (\oa_{n}^{(4)\star})_{n\in\Z} \colon H_{r}^{3/2}\to \ell_{\R}^{-3,\infty}$, referred to as the frequency map, is real analytic.

\begin{thm}
\label{thm:nls-freq}
The map $\oa^{(4)\star}$ is defined on $H_{r}^{0}=\FL_{r}^{2}$, takes values in $\bigcap_{r > 1} \ell^{-1,r}$,
and $\oa^{(4)\star}\colon \FL_{r}^{2}\to \ell^{-1,r}$ is real analytic for any $r > 1$.
Moreover, for any $p> 2$ the map $\oa^{(4)\star}$ admits a real analytic extension $\oa^{(4)\star}\colon \FL_{r}^{p}\to \ell_{\C}^{-1,p/2}$ with asymptotics
\[
  \oa_{n}^{(4)\star} + 12n\pi I_{n}
  = n(\ell_{n}^{p/3} + \ell_{n}^{1+}),
\]
which hold locally uniformly on $\FL_{r}^{p}$.
Here $\ell_{n}^{p/3}$ denotes a generic sequence $(a_{m})_{m\in\Z}$ with $\sum_{m\in\Z} \abs{a_{m}}^{p/3} < \infty$ and $\ell_{n}^{1+}$ denotes a generic sequence $(b_{m})_{m\in\Z}$ with $\sum_{m\in\Z} \abs{b_{m}}^{r} < \infty$ for any $r > 1$.~\fish
\end{thm}

Theorem~\ref{thm:wp-defocusing-mkdv} then follows by observing that the frequencies $\om_{n}^{\#}$ of the \rmKdV equation are obtained from the restriction of $\om_{n}^{(4)\star}$ to
$
  \Ec_{r}^{p} \defl \setd{\ph = (u,u)\in\FL_{r}^{p}}
$
as follows,
\[
  \om_{n}^{\#} \defl (2n\pi)^{3} + \om_{n}^{(4)\star}\Big|_{\Ec_{r}^{p}}.
\]
See Section~\ref{s:wp} for details.


\section{Preliminaries}
\label{s:preliminiaries}

Let us briefly recall the main properties of the Lax-pair formulation of the NLS system. For a potential $\ph=(\phm,\php)\in \FL_{c}^{p}$, $2 \le p < \infty$, consider the Zakharov-Shabat operator
\[
  L(\ph) \defl
  \mat[\bigg]{\,\ii & \\  & -\ii} 
  \frac{\ddd}{\dx} +
  \mat[\bigg]{ & \phm \\ \php & }
\]
on the interval $[0,2]$ with periodic boundary conditions. The spectral theory of this operator is well known in the case $p=2$ and has been extended to the case $p > 2$ in \cite{Molnar:2016uq}. One can show that the spectrum is discrete, and consists of a sequence of pairs of complex eigenvalues $\lm_n^+(\ph)$ and $\lm_n^-(\ph)$, when listed with algebraic multiplicities, such that
\[
  \lm_n^\pm(\ph) = n\pi + \ell^p_n,\qquad n\in\Z.
\]
Here $\ell_n^p$, $n\in\Z$, denotes a generic $\ell^p$-sequence. We may order the eigenvalues lexicographically -- first by their real part, and second by their imaginary part -- to represent them as a bi-infinite sequence of complex eigenvalues
\[
  \dotsb \lex \lambda_{n-1}^+ \lex \lambda_{n}^- \lex \lambda_{n}^+ \lex \lambda_{n+1}^- \lex \dotsb.
\]
By a slight abuse of terminology, we call the periodic eigenvalues of $L(\ph)$ the periodic spectrum of $\ph$. Further we introduce the gap lengths
\[
  \gm_n(\ph) \defl \lm_n^+ - \lm_n^- = \ell^p_n,\qquad n\in\Z,
\]
and the mid points
\[
  \tau_{n}(\ph) \defl \frac{\lm_{n}^{+}+\lm_{n}^{-}}{2} = n\pi + \ell_{n}^{p},\qquad n\in\Z.
\]

In the case of an $L^{2}_{c}$ potential, the periodic spectrum can be described in terms of the fundamental solution of $L(\ph)$ which is denoted by $M(x,\lm,\ph)$. In more detail, when one introduces the discriminant
\[
  \Dl(\lm,\ph) \defl  \operatorname{tr} M(1,\lm,\ph),
\]
then the periodic spectrum of $\ph$ is precisely the zero set of the entire function $\Dl^2(\lm,\ph) - 4$, and we have the product representation
\[
  \Dl^2(\lm,\ph) - 4
   = 
  -4\prod_{n\in\Z}
  \frac{(\lm_n^+-\lm)(\lm_n^--\lm)}{\pi_n^2},
  \qquad
  \pi_n
   \defl 
  \begin{cases}
  n\pi, & n\neq 0,\\
  1, & n=0.
  \end{cases}
\]
Further, $\Dl(\lm_{2n}^{+},\ph) = \Dl(\lm_{2n}^{-},\ph) = 2$, $\Dl(\lm_{2n+1}^{+},\ph) = \Dl(\lm_{2n+1}^{-},\ph) = -2$, for any $n\in\Z$, and
\begin{equation}
  \label{dl-inf-prod}
  \Dl(\lm,\ph)-2 = -\prod_{n\in\Z} \frac{(\lm_{2n}^+-\lm)(\lm_{2n}^--\lm)}{\pi_{2n}^2}.
\end{equation}

In the case of an $\FL_{c}^{p}$ potential with $2 < p <\infty$, we \emph{define} the discriminant by~\eqref{dl-inf-prod} so that we do not need to introduce the fundamental solution for such potentials. One can further show that in this case the $\lm$-derivative $\dDl\defl\partial_{\lm}\Dl$, whose zeros are denoted by $\lm_{n}^{\ld}$, $n\in\Z$, and are listed so that they have the asymptotics $\lm_{n}^{\ld} = n\pi + \ell^{p}_{n}$, admits the product representation
\[
  \dDl(\lm,\ph) = 2\prod_{n\in\Z} \frac{\lm_{n}^{\ld}-\lm}{\pi_{n}},
\]
as is known from the $L^{2}_{c}$ case. The asymptotics of $\lm_{n}^{\ld}$ and $\tau_{n}$ imply that $\lm_{n}^{\ld}-\tau_{n} = \ell_{n}^{p}$. As it turns out, $\lm_{n}^{\ld}$ is closer even to $\tau_{n}$, namely one has
\begin{equation}
  \label{lmld-tau-2}
  \lm_{n}^{\ld} = \tau_{n} + \gm_{n}^{2}\ell_{n}^{p},
\end{equation}
locally uniformly on $\FL_{c}^{p}$ for any $2 \le p < \infty$.

We also consider the Dirichlet spectrum of the Zakharov-Shabat operator -- see~\cite{Grebert:2014iq} for the definition in the case $p=2$ and~\cite{Molnar:2016uq} for an extension to the case $2 < p < \infty$. This spectrum, referred to as the \emph{Dirichlet spectrum of $\ph$}, is discrete and given by a sequence of eigenvalues $(\mu_{n})_{n\in\Z}$, counted with multiplicities and ordered lexicographically so that
\[
  \dotsb\lex \mu_{n-1} \lex \mu_{n} \lex \mu_{n+1} \lex \dotsb,\qquad \mu_{n}  = n\pi + \ell_{n}^{p}.
\]

If the potential $\ph$ is of real type, that is $\phm = \ob{\php}$, then periodic and Dirichlet spectra are real, and the lexicographical ordering coincides with the ordering of real numbers
\[
  \dotsb
   \le \lm_{n-1}^{+} < \lm_{n}^{-} \le \lm_{n}^{\ld},\mu_{n} \le \lm_{n}^{+} < \lm_{n+1}^{-}
   \le \dotsb.
\]
For each $\ph\in \FL_{r}^{p}$, there exists an open and connected neighborhood $V_{\ph}$ of $\ph$ within $\FL_{c}^{p}$ and disjoint closed discs $(U_{n})_{n\in\Z}$, centered on the real axis, so that the following holds
\begin{enumerate}[label=(S\arabic{*})]
\item
\label{iso-1}
 for any $n\in\Z$, $G_{n} \defl [\lm_{n}^{-},\lm_{n}^{+}]$, $\mu_{n}$, and $\lm_{n}^{\ld}$ are contained in the interior of $U_{n}$ for every $\psi\in V_{\ph}$;

\item there exists a constant $c \ge 1$ such that for all $n,m\in\Z$ with $m\neq n$
\begin{equation}
  \label{iso-est}
  c^{-1}\abs{m-n} \le \dist(U_{n},U_{m}) \le c\abs{m-n};
\end{equation}

\item
\label{iso-3}
there exists an integer $N_{\ph}\ge 1$ so that
\begin{equation}
  \label{Un-Dn}
  U_{n} = D_{n} \defl \setdef{\lm\in\C}{\abs{\lm - n\pi} \le \pi/5},\qquad \abs{n}\ge N_{\ph}.
\end{equation}
\end{enumerate}
Such discs $(U_{n})_{n\in\Z}$ are called isolating neighborhoods in the sequel. The union of all $V_{\ph}$, $\ph\in\FL_{r}^{p}$, defines an open and connected neighborhood of $\FL_{r}^{p}$ within $\FL_{c}^{p}$ and is denoted by $\Ws^{p}$. In the sequel, we denote for any $\ph\in\Ws^{p}$ an open and connected neighborhood $V_{\ph}$ satisfying the above conditions.

Following~\cite{Grebert:2014iq}, for $\ph \in \Wp^{p}$, one can define action variables for the NLS equation by
\begin{equation}
  \label{action}
  I_n
   =
  \frac{1}{\pi}\int_{\Gm_n}
  \frac{\lm \dDl(\lm)}{\sqrt[c]{\Dl^2(\lm)-4}}
  \,\dlm,\qquad n\in\Z.
\end{equation}
Here  $\Gm_{n}\subset U_{n}$ denotes any counter clockwise oriented circuit around and sufficiently close to $G_{n}$, and the \emph{canonical root} $\sqrt[c]{\Dl^{2}(\lm)-4}$ is defined on $\C\setminus \bigcup_{\gm_{n}\neq 0} G_{n}$, where, for $\ph$ of real type, the sign of the root is determined by
\[
  \ii\sqrt[c]{\Dl^{2}(\lm)-4} > 0,\qquad \lm_{0}^{+} < \lm < \lm_{1}^{-} \ ,
\]
and for $\ph\in\Wp^{p}$ it is defined by continuous extension.
According to \cite{Molnar:2016uq,Grebert:2014iq}, for any given $k\in\Z$ the action $I_{k}$ is a real analytic function on $\Wp^{p}$.

The  Dirichlet eigenvalues and the discriminant can be used to construct the angles $\th_{k}(\ph)$, $k\in\Z$, which are conjugated to the actions $I_{n}(\ph)$, $n\in\Z$. The angle $\th_{k}$ is defined modulo $2\pi$ on $\Wp^{p}\setminus Z_{k}$ and is a real analytic function on $\Wp^{p}\setminus Z_{k}$ when considered modulo $\pi$, where
\begin{equation}
  \label{Zk}
  Z_{k}\equiv Z_k^{p} \defl \setdef{\psi\in \Wp^{p}}{\gm_{k}^{2}(\psi) = 0}.
\end{equation}
Moreover, the following commutator relations hold for any $m,n\in\Z$
\begin{align}
  \label{e:commutators0}
  \pbr{I_m,I_n} = 0,\qquad
  \pbr{\th_n,I_{m}} = \dl_{nm},\qquad
  \pbr{\th_m,\th_n} = 0,
\end{align}
whenever the bracket is defined.

For any $\ph\in\FL_{r}^{p}\setminus Z_{k}$ with $k\in\Z$ define
\begin{equation}
  \label{e:z_k}
  z_{k}^{-}(\ph) \defl \sqrt[+]{I_k(\ph)}\, \e^{-\ii \th_k(\ph)},\qquad
  z_{k}^{+}(\ph) \defl \sqrt[+]{I_k(\ph)}\, \e^{\ii \th_k(\ph)}.
\end{equation}
(We point out that in \cite{Grebert:2014iq} the meaning of $z_{k}^{\pm}$ is a different one.)
It is shown in \cite{Molnar:2016uq,Grebert:2014iq} that the mappings $\FL_{r}^{p}\setminus Z_k\to{\C}$, $\ph\mapsto z_{k}^{\pm}(\ph)$ analytically extend to a neighborhood $\Wp^{p}$, (after possibly shrinking $\Ws^{p}$ if necessary).
The {\em Birkhoff map} is then defined as follows
\begin{equation}
  \label{Phi.def}
  \Phi_{p}\colon \Wp^{p} \to  \ell_{c}^{p} \quad
  \ph \mapsto \Phi_{p}(\ph) \defl \p*{z_{k}^{-}(\ph),z_{k}^{+}(\ph)}_{k\in\Z}
\end{equation}
Its main properties are stated in Theorem~\ref{bhf}.

\section{The abelian integral $F$}
\label{s:F-properties}

It is convenient to define for any $\ph\in\Ws^{p}$, $2 \le p < \infty$, the \emph{standard root}s
\[
  \vs_{n}(\lm) = \sqrt[\mathrm{s}]{(\lm_{n}^{+}-\lm)(\lm_{n}^{-}-\lm)},
                 \qquad \lm\in\C\setminus G_{n},\qquad n\ge 1,
\]
by the condition
\begin{align}
  \label{s-root}
  \vs_{n}(\lm) = (\tau_{n}-\lm)\sqrt[+]{1 - \gm_{n}^{2}/4(\tau_{n}-\lm)^{2}},
  						 \qquad \tau_{n} = (\lm_{n}^{-}+\lm_{n}^{+})/2.
\end{align}
Here $\sqrt[+]{\phantom{a}}$ denotes the principal branch of the square root on the complex plane minus the ray $(-\infty,0]$. The standard root is analytic in $\lm$ on $\C\setminus G_{n}$ and in $(\lm,\psi)$ on $(\C\setminus \ob{U_{n}})\times V_{\ph}$, and for all $n,m\ge 1$ with $n\neq m$ where $V_{\ph}$ and $c$ are as in~\ref{iso-1}-\ref{iso-3} above,
\begin{equation}
  \label{s-root-est}
  \inf_{\lm\in U_{n}}\abs{\vs_{m}(\lm)} \ge c^{-1}\abs{n-m}.
\end{equation}
If $\gm_{n} = 0$, then $\vs_{n}(\lm) = (\tau_{n}-\lm)$ is an entire function of $\lm$.
On the other hand, if $\gm_{n}\neq 0$, then $\vs_{n}$ extends continuously to both sides of $G_{n}$, denoted by $G_{n}^{\pm}$,
\begin{equation}
  \label{Gn-sides}
  G_{n}^{\pm} = \setdef{\lm_{t}^{\pm} = \tau_{n} + (t \pm \ii 0)\gm_{n}/2}{-1\le t\le 1},
\end{equation}
and we have
\begin{equation}
  \label{s-root-sides}
  w_{n}(\lm_{t}^{\pm}) = \mp \ii \frac{\gm_{n}}{2}\sqrt[+]{1-t^{2}},\qquad -1\le t\le 1.
\end{equation}

\begin{lem}
\label{s-root-reciprocal-estimate}
\label{int-wm-quot-est}
\begin{equivenum}
\item
Suppose $\gm_{n}\neq 0$ and $f$ is continuous on $G_{n}$, then
\[
  \sup_{\lm\in G_{n}^{+}\cup G_{n}^{-}}
  \abs*{\frac{1}{\pi}\int_{\lm_{n}^{-}}^{\lm} \frac{f(z)}{w_{n}(z)}\,\dz}
  \le
  \max_{\lm\in G_{n}} \abs{f(\lm)}.
\]
\item
Suppose $f$ is analytic in a neighborhood of $G_{n}$ containing $\Gm_{n}$, then
\[
  \frac{1}{2\pi}\abs*{\int_{\Gm_{n}} \frac{f(\lm)}{\vs_{n}(\lm)}\,\dlm} \le
  \max_{\lm\in G_{n}}\abs{f(\lm)}.
\]

\item
For any $n,m\in\Z$ we have
\[
  \frac{1}{2\pi\ii}\int_{\Gm_{m}} \frac{\dlm}{\vs_{n}(\lm)} = -\dl_{m,n}.\fish
\]
\end{equivenum}

\end{lem}

\begin{proof}
(i) We choose the parametrization $\lm_{t}^{\pm}$ of $G_{n}^{\pm}$, defined in~\eqref{Gn-sides}, to obtain for $-1\le t\le 1$,
\[
  \int_{\lm_{n}^{-}}^{\lm_{t}^{\pm}} \frac{f(z)}{w_{n}(z)}\,\dz
  =
  \pm \ii \int_{-1}^{t} \frac{f(\lm_{r}^{\pm})}{\sqrt[+]{1-r^{2}}}\,\dr.
\]
Since $\int_{-1}^{1} \frac{1}{\sqrt[+]{1-r^{2}}}\,\dr = \pi$, the claim follows immediately.

(ii) If $\gm_{n} = 0$, then $\vs_{n}(\lm) = \tau_{n}-\lm$ and the claim follows from Cauchy's theorem. On the other hand, if $\gm_{n}\neq 0$, then we may apply the previous lemma.

(iii)
First suppose $n\neq m$. Then $\vs_{n}$ is analytic on $U_{m}$ and therefore $\frac{1}{2\pi\ii}\int_{\Gm_{m}} \frac{\dlm}{\vs_{n}(\lm)}=0$. Now suppose $m=n$. If $\gm_{n} = 0$, then $\vs_{n}(\lm) = \tau_{n}-\lm$ and the claim follows from Cauchy's theorem. On the other hand, if $\gm_{n}\neq 0$, then in view of~\eqref{s-root-sides}
\[
  \frac{1}{2\pi\ii}\int_{\Gm_{n}} \frac{\dlm}{\vs_{n}(\lm)}
  =
  -\frac{2}{2\pi}\int_{-1}^{1} \frac{\dt}{\sqrt{1-t^{2}}}
  = -1.\qed
\]
\end{proof}

The \emph{canonical root} $\sqrt[c]{\Dl^{2}(\lm)-4}$ can be written in terms of standard roots as follows
\begin{equation}
  \label{c-root}
  \sqrt[c]{\Dl^{2}(\lm)-4} \defl
   2\ii\prod_{m\in\Z} \frac{\vs_{m}(\lm)}{\pi_{m}}
\end{equation}
and is analytic in $\lm$ on $\C\setminus\bigcup_{\gm_{m}\neq 0} G_{m}$ and in $(\lm,\psi)$ on $(\C\setminus \bigcup_{m\in\Z} \ob{U_{m}})\times V_{\ph}$. By Corollary~\ref{wm-ana-quot}, the quotient
\begin{equation}
  \label{om}
    \frac{\dDl(\lm)}{\sqrt[c]{\Dl^{2}(\lm)-4}}
  =
  -\ii\prod_{m\in\Z} \frac{\lm_{m}^{\ld} - \lm}{\vs_{m}(\lm)},
\end{equation}
is analytic in $(\lm,\psi)$ on $(\C\setminus \bigcup_{m\in\Z} \ob{U_{m}}) \times V_{\ph}$, and analytic in $\lm$ on $\C\setminus \bigcup_{\gm_{m}\neq 0} G_{m}$.

We call a path in the complex plane \emph{admissible} for $\ph\in \FL_{c}^{p}$ if, except possibly at its endpoints, it does not intersect any gap $G_{m}(\ph)$, $m\in\Z$.

\begin{lem}
\label{w-closed}
For each $\ph\in \Ws^{p}$, $2 \le p < \infty$, the following holds:
\begin{equivenum}
\item
The function $\frac{\dDl(\lm,\psi)}{\sqrt[c]{\Dl^{2}(\lm,\psi)-4}}$ extends analytically to $\C\setminus \bigcup_{\gm_{m}\neq 0} G_{m}$;
\item
for any admissible path from $\lm_{n}^{-}$ to $\lm_{n}^{+}$ in $U_{n}$,
\[
  \int_{\lm_{n}^{-}}^{\lm_{n}^{+}} \frac{\dDl(\lm,\psi)}{\sqrt[c]{\Dl^{2}(\lm,\psi)-4}}\,\dlm =  0.
\]
In particular, for any closed circuit $\Gm_{n}$ in $U_{n}$ around $G_{n}$,
\[
  \int_{\Gm_{n}} \frac{\dDl(\lm,\psi)}{\sqrt[c]{\Dl^{2}(\lm,\psi)-4}}\,\dlm = 0.\fish
\]
\end{equivenum}
\end{lem}

\begin{proof}
Standard -- see e.g. \cite{Molnar:2016uq}.\qed
\end{proof}

By Lemma~\ref{w-closed}, for any $\ph\in\Ws^{p}$, $2 \le p < \infty$, the quotient~\eqref{om} is analytic in both variables $(\lm,\ps)$ on $(\C\setminus \bigcup_{m\in\Z} \ob{U_{n}})\times V_{\ph}$ and analytic in $\lm$ on $\C\setminus\bigcup_{\gm_{m}\neq 0} G_{m}$. We proceed by defining on the same domain for any $n\in\Z$ the primitive
\[
  F_{n}(\lm,\psi)
  \defl
  \frac{1}{2}\p*{
  \int_{\lm_{n}^{-}(\psi)}^{\lm} \frac{\dDl(\mu,\ps)}{\sqrt[c]{\Dl^{2}(\mu,\ps)-4}}\,\dmu +
  \int_{\lm_{n}^{+}(\psi)}^{\lm} \frac{\dDl(\mu,\ps)}{\sqrt[c]{\Dl^{2}(\mu,\ps)-4}}\,\dmu
  },
\]
where the paths of integration are chosen to be admissible. These improper integrals exist, since for $\gm_{n} = 0$ the integrand  is analytic on $U_{n}$, while for $\gm_{n}\neq 0$ it is of the form $1/\sqrt{\lm_{n}^{\pm}-\lm}$ locally around $\lm_{n}^{\pm}$. Moreover, $\int_{\lm_{n}^{-}}^{\lm_{n}^{+}}\frac{\dDl(\lm)}{\sqrt[c]{\Dl^{2}(\lm)-4}}\,\dlm = 0$ by Lemma~\ref{w-closed}, hence the definition of $F_{n}$ is independent of the chosen admissible path and one has
\[
  F_{n}(\lm)
   = \int_{\lm_{n}^{-}(\psi)}^{\lm} \frac{\dDl(\mu)}{\sqrt[c]{\Dl^{2}(\mu)-4}}\,\dmu
   = \int_{\lm_{n}^{+}(\psi)}^{\lm} \frac{\dDl(\mu)}{\sqrt[c]{\Dl^{2}(\mu)-4}}\,\dmu.
\]
Even though the eigenvalues $\lm_{n}^{\pm}$ are, due to their lexicographical ordering, not even continuous on $\Ws^{p}$, the mappings $F_{n}$ turn out to be analytic.

\begin{lem}
\label{F-prop}
For every $\ph\in \Ws^{p}$, $2 \le p < \infty$, the following holds:

\begin{equivenum}

\item $F_{n}$ is analytic in both variables $(\lm,\ps)$ on $(\C\setminus\bigcup_{m\in\Z} \ob{U_{m}}) \times V_{\ph}$ with gradient
\[
  \partial F_{n}(\lm) = \frac{\partial\Dl(\lm)}{\sqrt[c]{\Dl^{2}(\lm)-4}},
\]
and $F_{n}(\lm)\equiv F_{n}(\lm,\ph)$ is analytic in $\lm$ on $\C\setminus \bigcup_{\gm_{m}\neq 0} G_{m}$.

\item $F_{0}(\lm) = F_{n}(\lm) - \ii n\pi$ on $\C\setminus \bigcup_{\gm_{m}\neq 0} G_{m}$. In particular, $F_{0}$ extends continuously to all points $\lm_{m}^{\pm}$, $m\in\Z$, and one has
\[
  F_{0}(\lm_{m}^{+}) = F_{0}(\lm_{m}^{-}) = - \ii m\pi,\qquad m\in\Z.
\]

\item For any $n\in\Z$, $F_{n}^{2}(\lm)$ is analytic on $U_{n}$.

\item
Locally uniformly in $\ph$ and uniformly as $\abs{n}\to\infty$,
\[
  \sup_{\lm \in G_{n}^{+}\cup G_{n}^{-}} \abs{F_{n}(\lm)} = O(\gm_{n}).
\]

\item 
$
  I_{n} = -\frac{1}{\pi}\int_{\Gm_{n}} F(\lm)\,\dlm
$
for every $n\in\Z$.

\item If $\ph$ is of real type, then for any $\lm\in G_{n}$
\[
  F_{n}(\lm\pm \ii 0)  = \pm \cosh^{-1} \frac{(-1)^{n}\Dl(\lm)}{2}.
\]

%
%
\item At the zero potential one has $F_{n}(\lm,0) = -\ii \lm + \ii n\pi$.~\fish

\end{equivenum}

\end{lem}

\begin{proof}
(i)
The proof of the analyticity of $F_{n}$ on $(\C\setminus \bigcup_{m\in\Z} \ob{U_{m}})\times V_{\ph}$ is standard but a bit technical and can be found in~\cite[Appendix~F]{Molnar:2016uq}. The analyticity of $F_{n}(\lm) = F_{n}(\lm,\ph)$ on $\C\setminus\bigcup_{\gm_{m}\neq 0} G_{m}$ follows immediately from the properties of $\frac{\dDl(\lm)}{\sqrt[c]{\Dl^{2}(\lm)-4}}$ obtained in Lemma~\ref{w-closed}.

To obtain the formula for the gradient, we first consider the case of $\ph$ being of real type,
Then $(-1)^{n}\Dl(\lm,\ph)\ge 2$ on $G_{n}$ and hence
\[
  \min_{\lm_{n}^{-}\le \lm \le \lm_{n}^{+}} (-1)^{n}\Dl(\lm,\ph) - \sqrt[+]{\Dl^{2}(\lm,\ph)-4} > 0.
\]
Thus, after possibly shrinking $V_{\ph}$, we can choose a circuit $\Gm_{n}$, which is contained in $U_{n}$, and an open neighborhood $U_{n}'$ of $\Gm_{n}$ so that $\Gm_{n}$ circles around $G_{n}$, $\ob{U_{n}'}\subset U_{m}\setminus G_{m}$ for any potential in $V_{\ph}$, and the real part of $(-1)^{n}\p*{\Dl(\lm,\ps) + \sqrt[c]{\Dl^{2}(\lm,\ps)-4}}$ is strictly positive on $U_{n}'$. In consequence, the principal branch of the logarithm
\[
  l_{n}(\lm,\ps) = \log \frac{(-1)^{n}}{2}\p*{\Dl(\lm,\ps) + \sqrt[c]{\Dl^{2}(\lm,\ps)-4}}
\]
is analytic on $U_{n}'\times V_{\ph}$. Clearly, $\partial_{\lm} l_{n} = \frac{\dDl(\lm)}{\sqrt[c]{\Dl^{2}(\lm)-4}}$ and $l_{n}(\lm_{n}^{\pm})\equiv 0$, hence $F_{n} = l_{n}$. Taking the gradient of the above identity yields on $U_{n}'\times V_{\ph}$
\[
  \partial F_{n} = \partial l_{n}(\lm) = \frac{\partial\Dl(\lm)}{\sqrt[c]{\Dl^{2}(\lm)-4}}.
\]
Since both hand sides of this identity are analytic in both variables on $(\C\setminus\bigcup_{m\in\Z}\ob{U_{m}})\times V_{\ph}$ and analytic in $\lm$ on $\C\setminus \bigcup_{\gm_{m}\neq 0} G_{m}$, the formula for the gradient extends to these domains by the identity theorem.

(ii)
Note that $F_{0}(\lm) = F_{n}(\lm) + \int_{\lm_{0}^{+}}^{\lm_{n}^{+}}\frac{\dDl(\lm)}{\sqrt[c]{\Dl^{2}(\lm)-4}}\,\dlm$. Since $\int_{\lm_{n}^{-}}^{\lm_{n}^{+}}\frac{\dDl(\lm)}{\sqrt[c]{\Dl^{2}(\lm)-4}}\,\dlm = 0$ for any $n\in\Z$ by Lemma~\ref{w-closed}, $F_{n}(\lm_{n}^{+}) = F_{n}(\lm_{n}^{-}) = 0$. Moreover,
\[
  \int_{\lm_{0}^{+}}^{\lm_{n}^{+}}\frac{\dDl(\lm)}{\sqrt[c]{\Dl^{2}(\lm)-4}}\,\dlm
   = \sum_{k=0}^{n-1} \int_{\lm_{k}^{+}}^{\lm_{k+1}^{-}} \frac{\dDl(\lm)}{\sqrt[c]{\Dl^{2}(\lm)-4}}\,\dlm,\quad n\ge 1,
\]
while for $n\le -1$
\[
  \int_{\lm_{0}^{+}}^{\lm_{n}^{+}}\frac{\dDl(\lm)}{\sqrt[c]{\Dl^{2}(\lm)-4}}\,\dlm
   = -\sum_{k=n}^{-1} \int_{\lm_{k}^{+}}^{\lm_{k+1}^{-}} \frac{\dDl(\lm)}{\sqrt[c]{\Dl^{2}(\lm)-4}}\,\dlm,\quad n\le -1.
\]
Therefore, it is to compute $\int_{\lm_{k}^{+}}^{\lm_{k+1}^{-}} \frac{\dDl(\lm)}{\sqrt[c]{\Dl^{2}(\lm)-4}}\,\dlm$ for $k\in\Z$. To do this, first consider the case where $\ph$ is of real type.
In this case $\ii(-1)^{k}\sqrt[c]{\Dl^{2}(\lm)-4} > 0$ for $\lm_{k}^{+} < \lm < \lm_{k+1}^{-}$ -- c.f. \cite[Section 5]{Grebert:2014iq} -- so
\begin{align*}
  \int_{\lm_{k}^+}^{\lm_{k+1}^-} \frac{\dDl(\lm)}{\sqrt[c]{\Dl^{2}(\lm)-4}}\,\dlm 
  &=
  \ii(-1)^{k}\int_{\lm_{k}^+}^{\lm_{k+1}^-}
  \frac{\dDl(\lm)}{\sqrt[+]{4-\Dl^2(\lm)}}\,\dlm\\
   &=
  \ii(-1)^{k}
  \sin^{-1}\frac{\Dl(\lm)}{2}\bigg|_{\lm_{k}^+}^{\lm_{k+1}^-}
   =
  -\ii \pi,
\end{align*}
and hence $\int_{\lm_{0}^{+}}^{\lm_{n}^{+}}\frac{\dDl(\lm)}{\sqrt[+]{4-\Dl^2(\lm)}}\,\dlm  = -\ii n\pi$ for any $n\in\Z$. The function $\int_{\lm_{0}^{+}}^{\lm_{n}^{+}}\frac{\dDl(\lm)}{\sqrt[+]{4-\Dl^2(\lm)}}\,\dlm = F_{0}(\lm) - F_{n}(\lm)$ is analytic on $V_{\ph}$ by item (i) and vanishes by the preceding argument on the real subspace $V_{\ph}\cap \FL_{r}^{p}$. Therefore, $\int_{\lm_{0}^{+}}^{\lm_{n}^{+}}\frac{\dDl(\lm)}{\sqrt[+]{4-\Dl^2(\lm)}}\,\dlm  = -\ii n\pi$ holds true on all of $V_{\ph}$ in view of the identity theorem.

(iii)
In view of item (i) it remains to show that $F_{n}^{2}$ admits also for $\gm_{n}\neq 0$ an analytic extension from $U_{n}\setminus G_{n}$ to all of $U_{n}$. Write~\eqref{om} in the form
\begin{equation}
  \label{chi-n}
  \frac{\dDl(\lm)}{\sqrt[c]{\Dl^{2}(\lm)-4}} = -\ii\frac{\lm_{n}^{\ld}-\lm}{\vs_{n}(\lm)}\chi_{n}(\lm),\qquad\;\;\;
  \chi_{n}(\lm) = \prod_{m\neq n} \frac{\lm_{m}^{\ld}-\lm}{\vs_{m}(\lm)}.
\end{equation}
The functionals $\chi_{n}$, $n\in \Z$, are analytic on $U_{n}$ by Corollary~\ref{wm-ana-quot}. Moreover, the roots $\vs_{n}(\lm)$, $n\in\Z$, admit opposite signs on opposite sides of $G_{n}$. Therefore, in view of $F_{n}(\lm) = \int_{\lm_{n}^{+}}^{\lm} \frac{\dDl(\mu)}{\sqrt[c]{\Dl^{2}(\mu)-4}}\,\dmu$, for any $\lm\in G_{n}$,
\[
  F_{n}\big|_{G_{n}^{+}}(\lm) = - F_{n}\big|_{G_{n}^{-}}(\lm).
\]
Consequently, $F_{n}^{2}$ is continuous and hence analytic on all of $U_{n}$.

(iv)
If $\gm_{n} = 0$, then $G_{n} = \setd{\lm_{n}^{\pm}}$ and $F(\lm_{n}^{\pm}) = 0$ so there is nothing to show. Thus suppose $\gm_{n}\neq 0$. We have $\lm_{n}^{\ld} = \tau_{n} + \gm_{n}^{2}\ell_{n}^{p}$ and $\sup_{\lm\in U_{n}}\abs{\chi_{n}(\lm) - 1} = \ell_{n}^{p}$ locally uniformly on $V_{\ph}$  by~\eqref{lmld-tau-2} and Lemma~\ref{wm-ana-quot}. In view of~\eqref{chi-n} it follows with Lemma~\ref{s-root-reciprocal-estimate} that
\[
  \sup_{\lm\in G_{n}^{-}\cup G_{n}^{+}}\abs{F_{n}(\lm)}
  \le
  O\p*{\sup_{\lm\in G_{n}} \abs{\lm_{n}^{\ld} - \lm}}
  = O(\gm_{n}),
\]
uniformly on $V$ and for all $n\in\Z$.

(v)
The claim follows from~\eqref{action} using integration by parts.

(vi)
If $\ph$ is of real type, then for any $\lm\in G_{n}$
\[
  F_{n}(\lm\pm \ii 0)
  = \pm \int_{\lm_{n}^{-}}^{\lm} \frac{(-1)^{n}\dDl(\mu)}{\sqrt[+]{\Dl^{2}(\mu)-4}} \dmu
  = \pm \cosh^{-1} \frac{(-1)^{n}\Dl(\lm)}{2}.
\]

(vii)
At the zero potential, $\left.\frac{\dDl(\lm)}{\sqrt[c]{\Dl^{2}(\lm)-4}}\right|_{\ph=0} = -\ii$, which follows directly from the product representation \eqref{om}. Integration thus yields $F_{n}(\lm,0) = -\ii(\lm - n\pi)$.\qed
\end{proof}

To simplify notation, write $F(\lm) \equiv F_{0}(\lm)$ and note that $F(\lm) = F_{n}(\lm) - \ii n\pi$ for any $n\in\Z$.

\begin{lem}
Suppose $\ph$ is a finite gap potential of real type, then $F$ is analytic outside a disc of finite radius centered at the origin and admits the Laurent expansion
\begin{equation}
  \label{F-exp}
  F(\lm) = -\ii \lm + \ii\sum_{n\ge 1} \frac{\Hm_{n}}{(2\lm)^{n}},
\end{equation}
where $\Hm_{n}$ denotes the $n$th Hamiltonian of the NLS hierarchy.~\fish
\end{lem}

\begin{proof}
By the preceding lemma $F(\lm)$ is analytic on $\C\setminus\bigcup_{\gm_{m}\neq 0} G_{m}$.
Suppose $\ph$ is a finite gap potential of real type, then $F(\lm)$ is analytic outside a disc of finite radius. Moreover, the product expansion~\eqref{om} of $\frac{\dDl(\lm)}{\sqrt[c]{\Dl^{2}(\lm)-4}}$ is finite,  whence one verifies directly that $F(\lm) = O(\lm)$ uniformly as $\abs{\lm}\to \infty$. Therefore, $F$ is meromorphic with a pole at infinity of order at most one, and it suffices to determine the Laurent expansion of $F$ along an arbitrary sequence $\nu_{n}$ with $\abs{\nu_{n}}\to \infty$.

Since $\ph$ is assumed to be of real type, given any $n\in\Z$, the restriction of the function $(-1)^{n+1} \Dl(\lm)$ to the interval $[\lm_{n}^{+},\lm_{n+1}^{-}]$ is strictly increasing from $-2$ to $2$. The canonical root for $\lm_n^+ < \lm < \lm_{n+1}^-$ is given by
\[
  \sqrt[c]{\Dl^2(\lm) - 4}
   = (-1)^{n+1} \ii \sqrt[+]{4 - \Dl^2(\lm)}.
\]
Furthermore, one computes for $\lm_n^+ < \lm < \lm_{n+1}^-$ that
\begin{align*}
  \partial_\lm \left(-\ii \asin\left((-1)^{n+1}\frac{\Dl(\lm)}{2}\right) \right)
   &= \ii \frac{(-1)^{n}\Dl^\ld(\lm)/2}{\sqrt[+]{1-\Dl^2(\lm)/4}}
   = \frac{\Dl^\ld(\lm)}{\sqrt[c]{\Dl^2(\lm) - 4}}.
\end{align*}
Hence for any $\lm_n^+ \le \lm \le \lm_{n+1}^-$,
\begin{align*}
  F(\lm)
  &= F(\lm_n^+) + \int_{\lm_n^+}^{\lm} \frac{\Dl^\ld(\mu)}{\sqrt[c]{\Dl^2(\mu) - 4}} \, \dmu \\
  &= -\ii n \pi + \left[-\ii \asin\left((-1)^{n+1}\frac{\Dl(\mu)}{2}\right) \right]_{\lm_n^+}^\lm \\
  &= -\ii (n + 1/2) \pi - \ii\asin\left((-1)^{n+1}\frac{\Dl(\lm)}{2}\right).
\end{align*}
Let $\nu_{n} = (n+1/2)\pi$, then by \cite[Theorem~4.8]{Grebert:2014iq}\footnote{Note that in in comparison to \cite{Grebert:2014iq} we multiplied for $n\ge 2$ the $n$th Hamiltonian with the factor $(-\ii)^{n+1}$.} for any $N\ge 1$
\[
  \Dl(\nu_{n}) = 2\cos\ii\sg_{N}(\nu_{n}) + O(\nu_{n}^{-N}),\qquad
  \sg_{N}(\lm) = -\ii \lm + \ii\sum_{n=1}^{N} \frac{\Hm_{n}}{(2\lm)^{n}},
\]
as $\abs{n}\to\infty$.
Using that $\partial_z \asin(z) = {1}/{\sqrt[+]{1-z^2}}$ one gets by the mean value theorem
\[
  \abs*{\asin\p*{ (-1)^{n+1}\frac{\Dl(\nu_{n})}{2} }
      - \asin\p*{ (-1)^{n+1} \cos \ii \sg_N(\nu_{n}) }} = O(\nu_{n}^{-N}),
\]
and hence
\[
  F(\nu_{n}) = -\ii \nu_{n} - \ii\asin\p*{ (-1)^{n+1} \cos\ii \sg_N(\nu_{n}) } + O(\nu_{n}^{-N}),
  \qquad n\to\infty.
\]

Finally, writing $(-1)^{n+1} = -\sin \nu_n$ one gets by the addition theorem for the sine
\[
  (-1)^{n+1} \cos \ii\sg_N(\nu_{n}) = \sin\p*{ \ii\sg_N(\nu_{n}) - \nu_n },
\]
and hence
\[
  -\ii\asin\p*{ (-1)^{n+1} \cos \ii \sg_N(\nu_{n}) } = \sg_N(\nu_{n}) + \ii\nu_{n}.
\]
This gives $F(\nu_{n}) = \sg_{N}(\nu_{n}) + O(\nu_{n}^{-N})$, hence the Laurent coefficients of $F$ can be determined from $\sg_{N}$.~\qed
\end{proof}

The expansion~\eqref{F-exp} of $F(\lm)$ yields one of $F^{4}(\lm)$ which we record for later use.

\begin{cor}
\label{cor:F3-F4}
Suppose $\ph$ is a finite gap potential of real type, then $F^{4}$ is analytic outside a disc of finite radius and admits the Laurent expansion
\begin{align*}
  F^{4}(\lm) &=
   \lm^{4}
   - 2\Hm_{1}\lm^{2}
   - \Hm_{2}\lm
   - \frac{1}{2}(\Hm_{3} - 3\Hm_{1}^{2})
   - \frac{1}{4}\p*{\Hm_{4} - 6\Hm_{1}\Hm_{2} }\frac{1}{\lm}
   + O(\lm^{-2}).~\fish
\end{align*}
\end{cor}

We conclude this section by refining the asymptotics $\sup_{\lm \in G_{n}^{+}\cup G_{n}^{-}}\abs{F_{n}(\lm)} = O(\gm_{n})$ of Lemma~\ref{F-prop}.

\begin{lem}
\label{Fn-refined}
For any $2 \le p < \infty$, locally uniformly on $\Ws^{p}$
\[
  \sup_{\lm\in G_{n}^{+}\cup G_{n}^{-}}\abs*{F_{n}(\lm) - \ii \vs_{n}(\lm)} = \gm_{n}(\ell_{n}^{p/2} + \ell_{n}^{1+}).
\]
At the zero potential, the right hand side of the latter identity vanishes, $F_{n}(\lm) = \ii\vs_{n}(\lm) = -\ii \lm + \ii n\pi$.~\fish
\end{lem}
\begin{proof}
With~\eqref{chi-n} write $F_{n}$ in the form
\[
  F_{n}(\lm)
   = \int_{\lm_{n}^{-}}^{\lm} \frac{\dDl(\mu)}{\sqrt[c]{\Dl^{2}(\mu)-4}}\,\dmu
   = -\ii\int_{\lm_{n}^{-}}^{\lm} \frac{\lm_{n}^{\ld}-\mu}{\vs_{n}(\mu)} \chi_{n}(\mu) \,\dmu,
  \qquad
  \chi_{n}(\lm) = \prod_{m\neq n} \frac{\lm_{m}^{\ld}-\lm}{\vs_{m}(\lm)}.
\]
By~\eqref{lmld-tau-2}, $\lm_{n}^{\ld}-\tau_{n} = \gm_{n}^{2}\ell_{n}^{p}$, hence Lemma~\ref{ana-quot-w-lp} gives
$\sup_{\lm\in U_{n}}\abs{\chi_{n}(\lm) - 1} = \ell_{n}^{p/2} + \ell_{n}^{1+}$. As an immediate consequence we obtain from Lemma~\ref{s-root-reciprocal-estimate} that
\[
  \sup_{\lm\in G_{n}^{+}\cup G_{n}^{-}}
  \abs*{F_{n}(\lm) - (-\ii)\int_{\lm_{n}^{-}}^{\lm} \frac{\lm_{n}^{\ld}-\mu}{\vs_{n}(\mu)} \,\dmu}
  \le \max_{\lm\in G_{n}^{+}\cup G_{n}^{-}} \abs*{(\lm_{n}^{\ld}-\lm)(\chi_{n}(\lm)-1)}
   = \gm_{n}(\ell_{n}^{p/2}+\ell_{n}^{1+}).
\]
One further checks that $\partial_{\lm} \vs_{n}(\lm) = -\frac{\tau_{n}-\lm}{\vs_{n}(\lm)}$ for $\lm\notin G_{n}$, hence
\[
  -\ii\int_{\lm_{n}^{-}}^{\lm} \frac{\lm_{n}^{\ld}-\xi}{\vs_{n}(\xi)} \,\dxi =
  \ii\vs_{n}(\lm) +
  \ii(\tau_{n}-\lm_{n}^{\ld})\int_{\lm_{n}^{-}}^{\lm} \frac{1}{\vs_{n}(\xi)} \,\dxi.
\]
If $\gm_{n} = 0$, then $\tau_{n} = \lm_{n}^{\ld}$ and the claim is evident.
On the other hand, if $\gm_{n}\neq 0$, then Lemma~\ref{s-root-reciprocal-estimate} gives $\sup_{\lm\in G_{n}^{-}\cup G_{n}^{+}}\abs*{\int_{\lm_{n}^{-}}^{\lm} \frac{1}{\vs_{n}(\xi)} \,\dxi} \le \pi$ and the claim follows with the estimate $\tau_{n}-\lm_{n}^{\ld} = \gm_{n}^{2}\ell_{n}^{p}$.\qed
\end{proof}

\section{\boldmath Refined estimate of $\lm_{m}^{\ld}-\sg_{m}^{n}$}
\label{s:refined-estimate}

For $\ph\in \Ws^{p}$, $2\le p < \infty$, we denote by $\psi_{n}$, $n\in\Z$, the entire function of the form
\begin{equation}
  \label{psi-form}
  \psi_{n}(\lm) = -\frac{2}{\pi_{n}}\prod_{k\neq n}\frac{\sg_{k}^{n}-\lm}{\pi_{k}},
  \qquad
  \sg_{k}^{n} = k\pi + \ell_{k}^{p},
\end{equation}
which is characterized by the property that
\begin{equation}
  \label{psi-int-eqn}
  \frac{1}{2\pi}\int_{\Gm_{k}} \frac{\psi_{n}(\lm)}{\sqrt[c]{\Dl^{2}(\lm)-4}}\,\dlm = \dl_{nk}.
\end{equation}
The roots of $\psi_{n}$ are precisely the complex numbers $\sg_{k}^{n}$, $k\neq n$, and can be shown to satisfy
\begin{equation}
  \label{psi-root-asymptotics}
  \sg_{k}^{n} - \tau_{k} = \gm_{k}^{2}\ell_{k}^{p}[n]
\end{equation}
uniformly in $n$ and locally uniformly on $\Ws^{p}$ -- see \cite{Molnar:2016uq}.
The expression $\ell_{k}^{p}[n]$, $k\in\Z$, means that there exists a sequence $\al_{k}^{n}$, $k\in\Z$, whose elements may depend on $n\in\Z$, such that $\sg_{k}^{n} = \tau_{k} + \gm_{k}^{2}\al_{k}^{n}$ where
\begin{equation}
  \label{ell-k-n}
  \sum_{k\in\Z} \abs{\al_{k}^{n}}^{p} \le C,
\end{equation}
and $C > 0$ can be chosen uniformly in $n$ and locally uniformly on $\Ws^{p}$.
Further, it turns out to be convenient to set
\begin{equation}
  \label{sgnn}
  \sg_{n}^{n} \defl \lm_{n}^{\ld}.
\end{equation}

According to~\eqref{psi-root-asymptotics} and~\eqref{lmld-tau-2}
\begin{equation}
  \label{sg-k-lm-k}
    \sg_{k}^{n}-\lm_{k}^{\ld}
   = (\sg_{k}^{n}-\tau_{k}) + (\tau_{k}-\lm_{k}^{\ld})
   = \gm_{k}^{2}\ell_{k}^{p}[n],
   \qquad k\neq n.
\end{equation}
The purpose of this subsection is to improve on these estimates.

\begin{prop}
\label{prop:sg-lm}
Locally uniformly on $\Ws^{p}$, $2 < p < \infty$,
\[
  \sg_{k}^{n}-\lm_{k}^{\ld} = \gm_{k}\ell_{k}^{1+}[n].
\]
In more detail, there exists a sequence $\al_{k}^{n}$ so that $\sg_{k}^{n} = \lm_{k}^{\ld} + \gm_{k}\al_{k}^{n}$ where for any $q > 1$
\[
  \sum_{k\in\Z} \abs{\al_{k}^{n}}^{q} \le C_{q},
\]
and $C_{q} > 0$ can be chosen uniformly in $n$ and locally uniformly on $\Ws^{p}$.\fish
\end{prop}

\begin{rem}
Note that in the case $p=2$, the estimate~\eqref{sg-k-lm-k} implies that $\sg_{k}^{n} - \lm_{k}^{\ld} = \gm_{k}\ell_{k}^{1}[n]$.\map
\end{rem}

\begin{proof}
In view of~\eqref{psi-form} and~\eqref{c-root} for any $n,k\ge 1$
\begin{equation}
  \label{psi-c-root-quot}
  \frac{\psi_{n}(\lm)}{\sqrt[c]{\Dl^{2}(\lm)-4}}
   =
   \frac{\ii}{\vs_{n}(\lm)} \frac{\sg_{k}^{n}-\lm}{\sg_{n}^{n}-\lm} \zt_{k}(\lm),\qquad
  \zt_{k}(\lm) = \prod_{m\neq k} \frac{\sg_{m}^{n}-\lm}{\vs_{m}(\lm)},
\end{equation}
where the function $\zt_{k}$ is analytic on $U_{k}$.
By~\eqref{psi-int-eqn}, the roots $\sg_{k}^{n}$, $k\neq n$, of $\psi_{n}$ are characterized by the equation
\begin{equation}
  \label{psi-1}
  0
   = \int_{\Gm_{k}}\frac{\psi_{n}(\lm)}{\sqrt[c]{\Dl^{2}(\lm)-4}}\,\dlm
   = \ii\int_{\Gm_{k}}\frac{\sg_{k}^{n}-\lm}{\sg_{n}^{n}-\lm} \frac{\zt_{k}(\lm)}{\vs_{k}(\lm)}\,\dlm,\qquad
   k\neq n.
\end{equation}
It implies that
\begin{equation}
  \label{zt-k-identity}
    (\sg_{k}^{n}-\lm_{k}^{\ld})
  \int_{\Gm_{k}}\frac{1}{\sg_{n}^{n}-\lm} \frac{\zt_{k}(\lm)}{\vs_{k}(\lm)}\,\dlm
  =
  \int_{\Gm_{k}}\frac{\lm - \lm_{k}^{\ld}}{\sg_{n}^{n}-\lm}
  \frac{\zt_{k}(\lm)}{\vs_{k}(\lm)}\,\dlm,
  \qquad k\neq n.
\end{equation}
This identity is the starting point for estimating $\sg_{k}^{n}-\lm_{k}^{\ld}$.
A key step in the proof of the claimed estimate is to rewrite this identity in an appropriate way.
Let us multiply it by $(\sg_{n}^{n}-\lm_{k}^{\ld})$ and introduce
\begin{equation}
  \label{xi-k}
  \xi_{k}(\lm) = \frac{\sg_{n}^{n}-\lm_{k}^{\ld}}{\sg_{n}^{n}-\lm}\zt_{k}(\lm).
\end{equation}
Note that $\sg_{n}^{n}-\lm_{k}^{\ld}\neq 0$ by~\eqref{iso-est}.
It then follows from~\eqref{zt-k-identity} that
\[
  (\sg_{k}^{n}-\lm_{k}^{\ld})
  \int_{\Gm_{k}}\frac{\xi_{k}(\lm)}{\vs_{k}(\lm)}\,\dlm
  =
  \int_{\Gm_{k}}\frac{(\lm - \lm_{k}^{\ld})\xi_{k}(\lm)}{\vs_{k}(\lm)}\,\dlm,
  \qquad k\neq n.
\]
Since by~\eqref{xi-k}, $\xi_{k}(\lm) = \p*{1 + \frac{\lm-\lm_{k}^{\ld}}{\sg_{n}^{n}-\lm}}\zt_{k}(\lm)$, we get
\[
  \int_{\Gm_{k}}\frac{(\lm - \lm_{k}^{\ld})\xi_{k}(\lm)}{\vs_{k}(\lm)}\,\dlm
  =
  \int_{\Gm_{k}}\frac{(\lm-\lm_{k}^{\ld})\zt_{k}(\lm)}{\vs_{k}(\lm)}\,\dlm
  +
  \int_{\Gm_{k}}\frac{(\lm-\lm_{k}^{\ld})^{2}\zt_{k}(\lm)}{(\sg_{n}^{n}-\lm)\vs_{k}(\lm)}\,\dlm.
\]
The second term on the right hand side of the latter identity is expected to be small in comparison to the first term since $\sg_{n}^{n}-\lm$ is of the size of $n-k$. We proceed by writing the first term in a more convenient form. Note that the roots $\lm_{k}^{\ld}$, $k\in\Z$, of $\dDl$ are characterized by the equation
\begin{equation}
  \label{dDl-1}
  0 = \int_{\Gm_{k}} \frac{\dDl(\lm)}{\sqrt[c]{\Dl^{2}(\lm)-4}}\,\dlm
  =
  -\ii
  \int_{\Gm_{k}} 
  \frac{\lm_{k}^{\ld}-\lm}{\vs_{k}(\lm)}\chi_{k}(\lm)\,\dlm,\qquad k\in\Z,
\end{equation}
where $\chi_{k}$ is given by~\eqref{chi-n}. Hence
\[
  \int_{\Gm_{k}}\frac{(\lm-\lm_{k}^{\ld})\zt_{k}(\lm)}{\vs_{k}(\lm)}\,\dlm
  =
  \int_{\Gm_{k}}\frac{(\lm-\lm_{k}^{\ld})(\zt_{k}(\lm)-\chi_{k}(\lm))}{\vs_{k}(\lm)}\,\dlm.
\]
Altogether, identity~\eqref{zt-k-identity} then reads
\begin{equation}
  \label{zt-k-identity-2}
  \begin{split}
  (\sg_{k}^{n}-\lm_{k}^{\ld})
  \int_{\Gm_{k}}\frac{\xi_{k}(\lm)}{\vs_{k}(\lm)}\,\dlm
  &=
  \int_{\Gm_{k}}\frac{(\lm-\lm_{k}^{\ld})(\zt_{k}(\lm)-\chi_{k}(\lm))}{\vs_{k}(\lm)}\,\dlm\\
  &\qquad+
  \int_{\Gm_{k}}\frac{(\lm-\lm_{k}^{\ld})^{2}\zt_{k}(\lm)}{(\sg_{n}^{n}-\lm)\vs_{k}(\lm)}\,\dlm,\qquad k\neq n.
    \end{split}
\end{equation}
The integrals in~\eqref{zt-k-identity-2} are now estimated separately.
First note that for $\lm\in G_{k}$ we have
\[
  \frac{\sg_{n}^{n}-\lm_{k}^{\ld}}{\sg_{n}^{n}-\lm} = 1 + \frac{\lm-\lm_{k}^{\ld}}{\sg_{n}^{n}-\lm},\qquad
  \frac{\lm-\lm_{k}^{\ld}}{\sg_{n}^{n}-\lm} = O\p*{\frac{\gm_{k}}{n-k}}.
\]
By Lemma~\ref{ana-quot-w-lp} we have $\zt_{k}\big|_{G_{k}} = 1 + \ell_{k}^{p}$, implying that
\[
  \xi_{k}\big|_{G_{k}} = (1+\ell_{k}^{p})\zt_{k}\big|_{G_{k}} = 1+\ell_{k}^{p}.
\]
Since by Lemma~\ref{int-wm-quot-est} (iii), $\frac{1}{2\pi\ii}\int_{\Gm_{k}} \frac{1}{\vs_{k}(\lm)}\,\dlm = -1$, we then conclude
\begin{equation}
  \label{xi-k-est-1}
  \frac{1}{2\pi\ii}\int_{\Gm_{k}}\frac{\xi_{k}(\lm)}{\vs_{k}(\lm)}\,\dlm = -1 + \ell_{k}^{p}.
\end{equation}
Concerning the integral $\int_{\Gm_{k}}\frac{(\lm-\lm_{k}^{\ld})(\zt_{k}(\lm)-\chi_{k}(\lm))}{\vs_{k}(\lm)}\,\dlm$, it follows from Lemma~\ref{int-wm-quot-est} that
\begin{equation}
  \label{xi-k-est-2}
  \abs*{\frac{1}{2\pi}\int_{\Gm_{k}}\frac{(\lm-\lm_{k}^{\ld})(\zt_{k}(\lm)-\chi_{k}(\lm))}{\vs_{k}(\lm)}\,\dlm}
  \le
  \abs{\gm_{k}}\abs{\zt_{k}(\lm)-\chi_{k}(\lm)}_{G_{k}}.
\end{equation}
Similarly, for the integral $\int_{\Gm_{k}}\frac{(\lm-\lm_{k}^{\ld})^{2}\zt_{k}(\lm)}{(\sg_{n}^{n}-\lm)\vs_{k}(\lm)}\,\dlm$, we get
\begin{equation}
  \label{xi-k-est-3}
  \abs*{\frac{1}{2\pi}\int_{\Gm_{k}}\frac{(\lm-\lm_{k}^{\ld})^{2}\zt_{k}(\lm)}{(\sg_{n}^{n}-\lm)\vs_{k}(\lm)}\,\dlm}
  \le
  \abs*{\frac{(\lm-\lm_{k}^{\ld})^{2}}{\sg_{n}^{n}-\lm}\zt_{k}(\lm)}_{G_{k}}
  = O\p*{\frac{\gm_{k}^{2}}{n-k}}.
\end{equation}
Since $\frac{1}{n-k} = \ell_{k}^{1+}[n]$ we conclude
\begin{equation}
  \label{gm-quot-est}
  O\p*{\frac{\gm_{k}^{2}}{n-k}} = \gm_{k}\ell_{k}^{1}[n].
\end{equation}
Inserting estimates~\eqref{xi-k-est-1}-\eqref{gm-quot-est} into~\eqref{zt-k-identity-2} yields
\begin{equation}
  \label{system-1}
  (\sg_{k}^{n}-\lm_{k}^{\ld})
  (1+\ell_{k}^{p})
  =
  O\p*{\gm_{k}\abs{\zt_{k}(\lm)-\chi_{k}(\lm)}_{G_{k}}}
   + \gm_{k}\ell_{k}^{1}[n].
\end{equation}

To estimate $\abs{\zt_{k}(\lm)-\chi_{k}(\lm)}_{G_{k}}$, write the product expansions~\eqref{psi-c-root-quot} and~\eqref{chi-n}, respectively, as $\zt_{k}(\lm) = f_{k}(\lm,\tilde\al^{1})$ and  $\chi_{k}(\lm) = f_{k}(\lm,\tilde\al^{0})$, where we have set $\al^{1} = (\sg_{m}^{n})$, $\al^{0} = (\lm_{m}^{\ld})$, and
\[
  f_{k}(\lm,\tilde\al) = \prod_{m\neq k}\frac{\al_{m}-\lm}{\vs_{m}(\lm)},\qquad \tilde \al_{n} = \al_{n} - n\pi.
\]
By Corollary~\ref{wm-ana-quot}, the function~$f_{k}$ is analytic on $(\C\setminus\bigcup_{m\neq k} G_{k})\times\ell_{\C}^{p}$ and by Lemma~\ref{ana-quot-w-lp} satisfies the estimate $\abs{f_{k}(\lm,\tilde\al)-1}_{G_{k}} = 1+\ell_{k}^{p}$ locally uniformly on $\ell_{\C}^{p}$. Thus we may write for any $\lm\in G_{k}$,
\[
  \zt_{k}(\lm)-\chi_{k}(\lm)
  = f_{k}(\lm,\tilde\al^{1})-f_{k}(\lm,\tilde\al^{0})
  = \int_{0}^{1} \sum_{m\neq k} \partial_{\al_{m}} f_{k}(\lm,\tilde\al^{t}) (\sg_{m}^{n}-\lm_{m}^{\ld})\,\dt,
\]
where $\al^{t} \defl (\al_{m}^{t}) = ((1-t) \sg_{m}^{n} + t \lm_{m}^{\ld})$. Since for any $m\neq k$ one has
\[
  \partial_{\tilde\al_{m}} f_{k}(\lm,\tilde \al) = \frac{1}{\al_{m}-\lm}f_{k}(\lm,\tilde \al),
\]
we conclude
\[
  \zt_{k}(\lm)-\chi_{k}(\lm)
  =
  \int_{0}^{1} f_{k}(\lm,\tilde \al^{t})\sum_{m\neq k} 
  \frac{\sg_{m}^{n}-\lm_{m}^{\ld}}{\al_{m}^{t}-\lm} \,\dt.
\]
By Lemma~\ref{ana-quot-w-lp}, we can choose $M > 0$ so that $\sup_{0\le t\le 1}\abs{f_{k}(\lm,\tilde \al^{t})}_{G_{k}} \le M$ for all $k\ge1$. Moreover, by the mean value theorem there exists a sequence $(\nu_{k})\subset \C$ with $\nu_{k}\in G_{k}$ such that
\begin{equation}
  \label{zt-k-chi-k-est}
  \abs{\zt_{k}(\lm)-\chi_{k}(\lm)}_{G_{k}} \le
  M\int_{0}^{1}\abs*{\sum_{m\neq k} 
  \frac{\sg_{m}^{n}-\lm_{m}^{\ld}}{\al_{m}^{t}-\nu_{k}}}\,\dt.
\end{equation}
Since $U_{m}$ is a disc and hence convex, we have $\al_{m}^{t}\in U_{m}$ for all $m\ge 1$. Moreover, $\nu_{k}\in G_{k}\subset U_{k}$ for all $k\ge 1$. Thus by~\eqref{iso-est}, there exists $c > 0$ so that
\[
  \abs{\al_{m}^{t}-\nu_{k}} \ge c\abs{m-k},\qquad m\neq k,\quad 0\le t\le 1.
\]

Write the estimate~\eqref{sg-k-lm-k} in the form $\sg_{m}^{n}-\lm_{m}^{\ld} = \gm_{m}\ell_{m}^{q_{1}}[n]$ with $q_{1} = p/2$ and suppose that for some $j\ge 1$
\[
  \sg_{m}^{n}-\lm_{m}^{\ld} = \gm_{m}\ell_{m}^{q_{j}}[n],
  \qquad 1 < q_{j} \le p/2,
\]
then $\sg_{m}^{n}-\lm_{m}^{\ld} = \ell_{m}^{r_{j}}[n]$ with $r_{j} = \frac{p}{p+q_{j}}q_{j}$.
It follows with Lemma~\ref{A-trans} from~\eqref{zt-k-chi-k-est} that
\[
  \abs{\zt_{k}(\lm)-\chi_{k}(\lm)}_{G_{k}} = \gm_{k}(\ell_{k}^{r_{j}}+\ell_{k}^{1+}).
\]
We conclude with~\eqref{system-1} that for all $k$ sufficiently large
\begin{equation}
  \label{sg-k-lm-k-dot-2}
  \begin{split}
  \sg_{k}^{n}-\lm_{k}^{\ld}
  &=
  \gm_{k}(\ell_{k}^{r_{j}}+\ell_{k}^{1+})
  +
  \gm_{k}\ell_{k}^{1}[n]
  = \gm_{k}(\ell_{k}^{q_{j+1}} + \ell_{k}^{1+})
  \end{split}
\end{equation}
where
\[
  q_{j+1} = r_{j} = \frac{p}{p+q_{j}}q_{j} = \frac{p}{j+1}.
\]
Thus, after finitely many iterations, $q_{j+1} \le 1$ and hence $\sg_{k}^{n}-\lm_{k}^{\ld} = \gm_{k}\ell_{k}^{1+}$ as claimed. By going through the arguments of the proof, one verifies that the estimate holds locally uniformly on $\Ws^{p}$.\qed

\end{proof}

\section{Proof of Theorem 1.3}
\label{s:mKdV-frequencies}

In this section we derive, a formula for the frequencies $\om_{n}^{(4)}$, $n\in\Z$, of the mKdV system~\eqref{mkdv-sys} which we then use to study their asymptotics as $\abs{n}\to\infty$. Our starting point is the following identity for the $n$th frequency
\[
  \om_{n}^{(4)} = \pbr{\Hm_{4},\th_{n}},
\]
which a priori holds on $H_{c}^{3/2}\cap (\Ws^{2}\setminus Z_{n})$.

It turns out to be convenient to introduce for any integers $n,k\in\Z$ and $m\ge 0$ the moments
\[
  \Om_{nk}^{(m)}
  \defl \int_{\Gm_{k}} \frac{F_{k}^{m}(\lm)\psi_{n}(\lm)}{\sqrt[c]{\Dl^{2}(\lm)-4}}\,\dlm.
\]
We recall from Section~\ref{s:refined-estimate} the product representation of the quotient $\psi_{n}(\lm)/\sqrt[c]{\Dl^{2}(\lm)-4}$ given by~\eqref{psi-c-root-quot}
\begin{equation}
  \label{zt-ztn}
    \frac{\psi_{n}(\lm)}{\sqrt[c]{\Dl^{2}(\lm)-4}}
  =
  \ii\frac{\zt_{n}(\lm)}{\vs_{n}(\lm)}
  =
  \frac{\sg_{k}^{n}-\lm}{\vs_{k}(\lm)}\zt_{k}^{n}(\lm),
  \qquad
  \zt_{k}^{n}(\lm) =
   \frac{\ii}{\sg_{n}^{n}-\lm}\zt_{k}(\lm),
\end{equation}
where we recall that $\sg_{n}^{n} = \lm_{n}^{\ld}$, and $\zt_{k}(\lm) = \prod_{m\neq k} \frac{\sg_{m}^{n}-\lm}{\vs_{m}(\lm)}$. Note that for any $k\in\Z$, the functions $\zt_{k}$ and $\zt_{k}^{n}$ are analytic on $U_{k}$.

\begin{lem}
\label{Om-nk-prop}
Given any $2 \le p < \infty$, the moments $\Om_{nk}^{(m)}\colon\Ws^{p}\to \C$ have the following properties:

\begin{equivenum}
\item
$\Om_{nk}^{(0)} = 2\pi\dl_{nk}$, for all $n,k\in\Z$.~\fish

\item
Each moment $\Om_{nk}^{(m)}$, $n,k\in\Z$, $m\ge 1$, is analytic on $\Ws^{p}$.

\item
$\Om_{nk}^{(2l+1)} = 0$ for all $n,k\in\Z$ and $l\ge 0$.

\item
If $\gm_{k} = 0$, then $\Om_{nk}^{(m)} = 0$ for all $n\in\Z$ and $m\ge 1$.

\end{equivenum}
\end{lem}

\begin{proof}
(i): The identity follows from the characterization~\eqref{psi-int-eqn} of the functions $\psi_{n}$.

(ii):
Let $\ph\in\Ws^{p}$. We choose circuits $\Gm_{k}$, $k\in\Z$, and open neighborhoods $U_{k}'$ of $\Gm_{k}$ such that $\Gm_{k}$ circles around $G_{k}$ and $\ob{U_{k}'}\subset U_{k}\setminus G_{k}$ for any potential in $V_{\ph}$ with $V_{\ph}$ given as in Section~\ref{s:preliminiaries}. In view of~\eqref{c-root}, Lemma~\eqref{psi-form}, and Lemma~\ref{F-prop}, the integrand $\frac{F_{k}^{m}(\lm)\psi_{n}(\lm)}{\sqrt[c]{\Dl^{2}(\lm)-4}}$ is analytic on $U_{k}'\times V_{\ph}$ for any $k\in\Z$. Consequently, $\Om_{nk}^{(m)}$ is analytic on $V_{\ph}$.

(iii):
The function $F_{k}$ and the canonical root each extend continuously to the two sides $G_{k}^{\pm}$ of the gap $G_{k}$ and take opposite signs there. Consequently, for any $l\ge 0$ the quotient $F_{k}^{2l+1}(\lm)/\sqrt[c]{\Dl^{2}(\lm)-4}$ extends continuously from $U_{k}\setminus G_{k}$ to $U_{k}$ and hence is analytic on all of $U_{k}$.
Together with the fact that $\psi_{n}$ is an entire function, we conclude that $\Om_{nk}^{(2l+1)} = 0$ for all $n,k\in\Z$.

(iv): In view of item (iii) it remains to consider the case $m=2l$, $l\ge 1$, and $n,k\in\Z$ with $\gm_{k} = 0$. 
Suppose $k\neq n$. Since $\sg_{k}^{n} = \tau_{k}$ by \eqref{psi-root-asymptotics}, and $\vs_{k}(\lm) = \tau_{k}-\lm$ in view of~\eqref{s-root}, by~\eqref{psi-c-root-quot} the quotient $\psi_{n}(\lm)/\sqrt[c]{\Dl^{2}(\lm)-4}$ equals $\zt_{k}^{n}(\lm)$ and hence is analytic on $U_{k}$. Since by Lemma~\ref{F-prop} also $F_{k}^{2l}$ is analytic on $U_{k}$, we conclude $\Om_{nk}^{(2l)} = 0$.
Now suppose $k=n$. Since $\zt_{n}$ and $F^{2l}$ are analytic on $U_{n}$ and $\vs_{n}(\lm) = \tau_{n}-\lm$, we have in view of~\eqref{psi-c-root-quot} and Cauchy's Theorem that $\Om_{nn}^{(2)} = 2\pi F_{n}^{2l}(\tau_{n}) \zt_{n}(\tau_{n})$. Since $\gm_{n}=0$ we find by Lemma~\ref{F-prop} that $F_{n}(\tau_{n}) = F_{n}(\lm_{n}^{\pm}) = 0$ proving the claim.~\qed
\end{proof}

\begin{lem}
\label{om-Om-4}
For any finite gap potential of real type and any $n\in\Z$
\begin{equation}
  \label{om-star-Om-nk-4}
  \om_{n}^{(4)\star} = \om_{n}^{(4)} - 6\Hm_{2} - 12n\pi \Hm_{1} - (2n\pi)^{3} = -12\sum_{k\in\Z} k\Om_{nk}^{(2)}.~\fish
\end{equation}
\end{lem}

\begin{proof}
Suppose $\ph$ is a finite gap potential, then there exists $K\ge 1$ so that $\gm_{k}(\ph) = 0$ for $\abs{k} > K$.
By Corollary~\ref{cor:F3-F4}, the function $F^{4}(\lm)$ is analytic outside a sufficiently large circle $C_{r}$, which encloses all open gaps $G_{k}$, $\abs{k}\le K$, and admits the Laurent expansion
\[
  F^{4}(\lm) = \lm^{4}
   - 2\Hm_{1}\lm^{2}
   - \Hm_{2}\lm
   - \frac{1}{2}(\Hm_{3} - 3\Hm_{1}^{2})
   - \frac{1}{4}\p*{\Hm_{4} - 6\Hm_{1}\Hm_{2} }\frac{1}{\lm}
   + O(\lm^{-2})
\]
Therefore, by the Cauchy Theorem
\[
  -\frac{1}{4}\p*{\Hm_{4} - 6\Hm_{1}\Hm_{2}} = \frac{1}{2\pi\ii}\int_{C_{r}} F^{4}(\lm)\,\dlm.
\]
Consider any $n\in \Z$ with $\gm_{n}(\ph)\neq 0$. Then $\th_{n}\mod \pi$ is analytic near $\ph$. Since $\Hm_{1} = \sum_{m\in\Z} I_{m}$ and $\Hm_{2} = \sum_{m\in\Z} (2m\pi) I_{m}$, one computes that
\begin{align*}
  \om_{n}^{(4)} - 6\Hm_{2} - 12n\pi \Hm_{1} &= \pbr{\Hm_{4} - 6\Hm_{1}\Hm_{2},\th_{n}}\\
  &=
  \frac{4\ii}{2\pi}\int_{C_{r}} \pbr{F^{4}(\lm),\th_{n}}\,\dlm\\
  &=
  \frac{16\ii}{2\pi}\int_{C_{r}} \frac{F^{3}(\lm)\pbr{\Dl(\lm),\th_{n}}}{\sqrt[c]{\Dl^{2}(\lm)-4}}\,\dlm.
\end{align*}
Using $-2\pbr{\Dl(\lm),\th_{n}} = \ps_{n}(\lm)$ -- c.f. \cite[Lemma 18.2]{Kappeler:2003up} -- one thus obtains
\[
  \om_{n}^{(4)} - 6\Hm_{2} - 12n\pi \Hm_{1}
  =
  -\frac{8\ii}{2\pi}\int_{C_{r}} \frac{F^{3}(\lm)\psi_{n}(\lm)}{\sqrt[c]{\Dl^{2}(\lm)-4}}\,\dlm.
\]
On the one hand $F^{3}(\lm)$ is analytic on $\C\setminus\bigcup_{\gm_{k}\neq 0} G_{k}$, while on the other hand for any $k\in\Z$ with $\gm_{k} = 0$ one has $\sg_{k}^{n} = \tau_{k}$ and $\vs_{k}(\lm) = \tau_{k}-\lm$ so that in view of the product representations~\eqref{c-root} and~\eqref{psi-form} the integrand extends analytically to $U_{k}$. Consequently, the integrand is analytic on $\C\setminus\bigcup_{\abs{k}\le K} G_{k}$ and one obtains by contour deformation
\[
  \om_{n}^{(4)} - 6\Hm_{2} - 12n\pi \Hm_{1}
   = -\frac{8\ii}{2\pi} \sum_{\abs{k} \le K}
     \int_{\Gm_{k}} \frac{F^{3}(\lm)\psi_{n}(\lm)}{\sqrt[c]{\Dl^{2}(\lm)-4}}\,\dlm.
\]
Proceeding by expanding $F(\lm)^{3} = (F_{k}(\lm)-\ii k\pi)^{3} = F_{k}^{3}(\lm) - 3\ii (k\pi) F_{k}^{2}(\lm) - 3(k\pi)^{2} F_{k}(\lm) + \ii(k\pi)^{3}$ and using that $\Om_{nk}^{(3)} \equiv \Om_{nk}^{(1)} \equiv 0$ by Lemma~\ref{Om-nk-prop} we thus get
\begin{align*}
  \om_{n}^{(4)} - 6\Hm_{2} - 12n\pi \Hm_{1}
   &= -\frac{8}{2\pi} \sum_{\abs{k} \le K}
   \p*{3(k\pi)\Om_{nk}^{(2)} - (k\pi)^{3}\Om_{nk}^{(0)}}\\
   &= \sum_{k\in\Z} \p*{-12k\Om_{nk}^{(2)} + (2k\pi)^{3}\dl_{kn}},
\end{align*}
where we used in the second line that $\Om_{nk}^{(2)} = 0$ for all $k > K$.
This shows that~\eqref{om-star-Om-nk-4} holds for all $n$ with $\gm_{n}(\ph)\neq 0$.

Now consider any $n\in\Z$ with $\gm_{n}(\ph)=0$, and denote $A = \setdef{k\in\Z}{\gm_{k}(\ph)\neq 0}$. By~\cite{Grebert:2014iq} we can choose a sequence of finite gap potentials $\ph_{l}$ in $H_{r}^{1}$ with $\gm_{k}(\ph_{l}) = \gm_{k}(\ph)$ for $k\neq n$, $\gm_{n}(\ph_{l})\neq 0$, and $\ph_{l}\to \ph$ in $H_{r}^{1}$. In particular, $A^{(l)} \defl \setdef{j\in\Z}{\gm_{j}(\ph_{l})\neq 0}$ is given by $A \cup \setd{n} \defr A^{\natural}$ for any $l\ge 1$. Since each $\Om_{nk}^{(2)}$, $k\in\Z$, is continuous, indeed analytic, on $\Ws^{p}$, and $A^{\natural}$ is finite and independent of $l$, it follows that
\[
  \sum_{k\in\Z} k\Om_{nk}^{(2)}(\ph_{l})
   =
  \sum_{k\in A^{\natural}} k\Om_{nk}^{(2)}(\ph_{l})
   \overset{l\to\infty}{\longrightarrow}
  \sum_{k\in\Z} k\Om_{nk}^{(2)}(\ph)
   =
  \sum_{k\in A^{\natural}} k\Om_{nk}^{(2)}(\ph).
\]
On the other hand, $\op_{n}$ is continuous at $\ph\in H_{r}^{2}$ which shows that~\eqref{om-star-Om-nk-4} holds for all $n\in\Z$.~\qed
\end{proof}

We proceed by deriving decay estimates for $\Om_{nk}^{(2)}$.

\begin{lem}
\label{decay-Omnk2}
Locally uniformly on $\Ws^{p}$, $2 \le p < \infty$,
\[
  \Om_{nk}^{(2)}
  = \frac{\gm_{k}^{3}\ell_{k}^{1+}[n]}{n-k},\quad k\neq n,\qquad
  \Om_{nn}^{(2)}
  = \frac{\gm_{n}^{2}}{4}\p*{\pi + \ell_{n}^{p/2} + \ell_{n}^{1+}}.
\]
In more detail, there exists a sequence $\al_{k}^{n}$, $k\in\Z$, so that $\Om_{nk}^{(2)} = \frac{\gm_{k}^{3}}{n-k}\al_{k}^{n}$ where for any $q > 1$
\[
  \sum_{k\in\Z} \abs{\al_{k}^{n}}^{q} \le C_{q},
\]
and $C_{q} > 0$ can be chosen uniformly in $n$ and locally uniformly on $\Ws^{p}$.\fish
\end{lem}

\begin{proof}
We begin with the case $k\neq n$ and prove the estimate
\begin{equation}
  \label{Omnk2-est-1}
  (\sg_{n}^{n}-\tau_{k})\Om_{nk}^{(2)} = \gm_{k}^{3}\ell_{k}^{1}[n]
\end{equation}
using Proposition~\ref{prop:sg-lm}. Since by~\eqref{iso-est}, $\abs{\sg_{n}^{n}-\tau_{k}}\ge c\abs{n-k}$, the claimed estimate $(n-k)\Om_{nk}^{(2)} = \gm_{k}^{3}\ell_{k}^{1}[n]$ then follows immediately. To prove~\eqref{Omnk2-est-1} note that
\[
  (\sg_{n}^{n}-\tau_{k})\Om_{nk}^{(2)}
  =
  \int_{\Gm_{k}} \frac{F_{k}^{2}(\lm) (\sg_{n}^{n}-\tau_{k})\psi_{n}(\lm) }{\sqrt[c]{\Dl^{2}(\lm)-4}}\,\dlm.
\]
Our goal is to compare the two product expansions
\[
  \frac{(\sg_{n}^{n}-\lm)\psi_{n}(\lm) }{\sqrt[c]{\Dl^{2}(\lm)-4}}
   =
  \ii\prod_{m\in \Z} \frac{\sg_{m}^{n}-\lm}{\vs_{m}(\lm)},
  \quad
  \frac{\dDl(\lm)}{\sqrt[c]{\Dl^{2}(\lm)-4}}
   =
  -\ii\prod_{m\in \Z} \frac{\lm_{m}^{\ld}-\lm}{\vs_{m}(\lm)}
\]
By Lemma~\ref{F-prop} the function $F_{k}^{3}(\lm)$ is analytic on $U_{k}\setminus G_{k}$ and we compute
\[
  \partial_{\lm}\p*{ \frac{1}{3} F_{k}^{3}(\lm) }
   = \frac{F_{k}^{2}(\lm)\dDl(\lm)}{\sqrt[c]{\Dl^{2}(\lm)-4}}.
\]
Therefore, in view of existence of the primitive
\[
  \int_{\Gm_{k}} \frac{F_{k}^{2}(\lm)\dDl(\lm)}{\sqrt[c]{\Dl^{2}(\lm)-4}}\,\dlm = 0.
\]
Thus the moment $\Om_{nk}^{(2)}$ may be written in the form
\begin{align*}
  (\sg_{n}^{n}-\tau_{k})\Om_{nk}^{(2)}
  &=
  \int_{\Gm_{k}} \frac{F_{k}^{2}(\lm)\p*{ (\sg_{n}^{n}-\tau_{k})\psi_{n}(\lm) + \dDl(\lm)}}{\sqrt[c]{\Dl^{2}(\lm)-4}}\,\dlm
  =
  I_{k}^{\flat}+ I_{k}^{\natural},
\end{align*}
where the terms $I_{k}^{\flat}$ and $I_{k}^{\natural}$ are defined by
\begin{align*}
   I_{k}^{\flat} &= \int_{\Gm_{k}} \frac{F_{k}^{2}(\lm)\p*{ (\sg_{n}^{n}-\lm)\psi_{n}(\lm) + \dDl(\lm)}}{\sqrt[c]{\Dl^{2}(\lm)-4}}\,\dlm,\\
  I_{k}^{\natural} &= 
  \int_{\Gm_{k}} \frac{F_{k}^{2}(\lm)(\lm-\tau_{k})\psi_{n}(\lm)}{\sqrt[c]{\Dl^{2}(\lm)-4}}\,\dlm.
\end{align*}

We first consider the term $I_{k}^{\natural}$. For $k\neq n$ we have by~\eqref{zt-ztn}
\[
  \frac{\psi_{n}(\lm)}{\sqrt[c]{\Dl^{2}(\lm)-4}} = \frac{\sg_{k}^{n}-\lm}{\vs_{k}(\lm)}\zt_{k}^{n}(\lm).
\]
By Lemma~\ref{ana-quot-w-lp} one has
\[
  (n-k)\zt_{k}^{n}(\lm) = \frac{\ii}{\pi} + \ell_{k}^{p/2}[n] + \ell_{k}^{1+}[n].
\]
Since $\abs{F_{k}}_{G_{k}} = O(\gm_{k})$ by Lemma~\ref{F-prop} and $\abs{\sg_{k}^{n}-\lm}_{G_{k}} = O(\gm_{k})$, we obtain with Lemma~\ref{int-wm-quot-est} that
\begin{align*}
  I_{k}^{\natural} = \int_{\Gm_{k}} \frac{F_{k}^{2}(\lm)(\lm-\tau_{k})\psi_{n}(\lm)}{\sqrt[c]{\Dl^{2}(\lm)-4}}\,\dlm
  &=
  \int_{\Gm_{k}} \frac{F_{k}^{2}(\lm)(\lm-\tau_{k})(\sg_{k}^{n}-\lm)\zt_{k}^{n}(\lm)}{\vs_{k}(\lm)}\,\dlm\\
  &= \frac{\gm_{k}^{4}}{n-k}\p*{\frac{\ii}{\pi} + \ell_{k}^{p/2}[n] + \ell_{k}^{+1}[n]}.
\end{align*}
Clearly, $I_{k}^{\natural} = \gm_{k}^{4}\ell_{k}^{1+}[n]$.

It remains to consider the term $I_{k}^{\flat}$. To this end, recall the product expansions
\[
  \dDl(\lm)
  =
  2\prod_{m\in\Z} \frac{\lm_{m}^{\ld}-\lm}{\pi_{m}},\quad
  \psi_{n}(\lm)
  =
  -\frac{2}{\pi_{n}}\prod_{m\neq n} \frac{\sg_{m}^{n}-\lm}{\pi_{m}},\quad
  \sqrt[c]{\Dl^{2}(\lm)-4}
  =
  2\ii\prod_{m\in \Z} \frac{\vs_{m}(\lm)}{\pi_{m}}.
\]
We may thus write
\[
  -\ii\frac{(\sg_{n}^{n}-\lm)\psi_{n}(\lm) + \dDl(\lm)}{\sqrt[c]{\Dl^{2}(\lm)-4}}
  =
  f(\lm,\tilde\al^{1}) - f(\lm,\tilde\al^{0}),
\]
where $\al^{1} = (\sg_{m}^{n})_{m\in\Z}$, $\al^{0} = (\lm_{m}^{\ld})_{m\in\Z}$, and
\[
  f(\lm,\tilde\al) \defl \frac{\al_{k}-\lm}{\vs_{k}(\lm)}f_{k}(\lm,\tilde\al),\quad
  f_{k}(\lm,\tilde\al) \defl \prod_{m\neq k} \frac{\al_{m}-\lm}{\vs_{m}(\lm)},\quad
  \al_{m} = m\pi + \tilde\al_{m}.
\]
By Corollary~\ref{wm-ana-quot} the functions
$f_{k}\colon (\C\setminus\bigcup_{m\neq k} G_{m})\times \ell_{\C}^{p}\to \C$ and $f\colon (\C\setminus\bigcup_{m\in\Z} G_{m})\times \ell_{\C}^{p}\to \C$ are analytic. One further computes that
\[
  \partial_{\tilde\al_{m}} f(\lm,\tilde\al) =  \frac{\al_{k}-\lm}{\al_{m}-\lm}\frac{f_{k}(\lm,\tilde \al)}{\vs_{k}(\lm)},\quad
  m\neq k,\qquad
  \partial_{\tilde\al_{k}} f(\lm,\tilde\al) =  \frac{f_{k}(\lm,\tilde\al)}{\vs_{k}(\lm)} .
\]
Further, for $0\le s\le 1$, let $\al^{s} = (\al_{m}^{s}) = ((1-s)\sg_{m}^{n} + s\lm_{m}^{\ld})$, then
\begin{align*}
  f(\lm,\tilde\al^{1}) - f(\lm,\tilde\al^{0})
  &=
  \int_{0}^{1}\sum_{m}
  \partial_{\al_{m}} f(\lm,\tilde\al^{s})(\sg_{m}^{n}-\lm_{m}^{\ld})\,\ds\\
  &=
  \int_{0}^{1}
  \p*{\sum_{m\neq k} \frac{\sg_{m}^{n}-\lm_{m}^{\ld}}{\al_{m}^{s}-\lm}}
  \frac{(\al_{k}^{s}-\lm)f_{k}(\lm,\tilde\al^{s})}{\vs_{k}(\lm)}\,\ds\\
  &\qquad +
  (\sg_{k}^{n}-\lm_{k}^{\ld})
  \int_{0}^{1}
  \frac{f_{k}(\lm,\tilde\al^{s})}{\vs_{k}(\lm)}\,\ds.
\end{align*}
Since first $\abs{f_{k}(\lm,\tilde\al^{s})}_{G_{k}}$ is bounded uniformly in $k$ and $0\le s\le 1$ by Lemma~\ref{ana-quot-w-lp}, second $\abs{\al_{k}^{s}-\lm}_{G_{k}} = O(\gm_{k})$ uniformly in $k$ and $0\le s \le 1$, third by~\eqref{iso-est},
\[
  \inf_{\lm\in G_{k}}\abs{\al_{m}^{s}-\lm} \ge c\abs{m-k},\qquad m\neq k,\quad 0\le s\le 1,
\]
fourth $\sg_{m}^{n}-\lm_{m}^{\ld} = \gm_{m}\ell_{m}^{1+}[n]$ by Proposition~\ref{prop:sg-lm}, and fifth $\abs{F_{k}} = O(\gm_{k})$ by Lemma~\ref{F-prop}, we obtain with Lemma~\ref{A-trans} and Lemma~\ref{int-wm-quot-est} that
\[
  I_{k}^{\flat} = \int_{\Gm_{k}}\frac{F_{k}^{2}(\lm)\p*{(\sg_{n}^{n}-\lm)\psi_{n}(\lm) + \dDl(\lm)}}{\sqrt[c]{\Dl^{2}(\lm)-4}} = \gm_{k}^{3}\ell_{k}^{1+}[n].
\]
We thus have shown the claimed estimate~\eqref{Omnk2-est-1}.

It remains to consider the case $k=n$. 
Since $\abs{F_{n}(\lm)}_{G_{n}}$, $\abs{w_{n}(\lm)}_{G_{n}} = O(\abs{\gm_{n}})$, it then follows from Lemma~\ref{Fn-refined} that
\[
  \abs{F_{n}^{2}(\lm)+\vs_{n}^{2}(\lm)}_{G_{n}} = \gm_{n}^{2}(\ell_{n}^{p/2} + \ell_{n}^{1+}).
\]
Moreover, $\zt_{n}(\lm)\big|_{U_{n}} = 1 + \ell_{n}^{p/2} + \ell_{n}^{1+}$ by Lemma~\ref{ana-quot-w-lp}. We thus write
\begin{align*}
  \Om_{nn}^{(2)} 
  &= \ii\int_{\Gm_{n}} \frac{F_{n}^{2}(\lm)\zt_{n}(\lm)}{\vs_{n}(\lm)}\,\dlm\\
  &= \ii\int_{\Gm_{n}} \frac{(-\vs_{n}^{2}(\lm) + F_{n}^{2}(\lm) +\vs_{n}^{2}(\lm))(1+\zt_{n}(\lm)-1)}{\vs_{n}(\lm)}\,\dlm,
\end{align*}
and may apply Lemma~\ref{int-wm-quot-est} to obtain the estimate
\begin{align*}
  &\frac{1}{2\pi}\abs*{
  \Om_{nn}^{(2)}  + \ii\int_{\Gm_{n}} \vs_{n}(\lm)\,\dlm}\\
  &\qquad\le
  \abs{F_{n}^{2}(\lm)+\vs_{n}^{2}(\lm)}_{G_{n}}
  + \abs{\vs_{n}^{2}(\lm)}_{G_{n}}\abs{\zt_{n}(\lm)-1}_{G_{n}}
  + \abs{F_{n}^{2}(\lm)+\vs_{n}^{2}(\lm)}_{G_{n}}\abs{\zt_{n}(\lm)-1}_{G_{n}}\\
  &\qquad = \gm_{n}^{2}(\ell_{n}^{p/2} + \ell_{n}^{1+}).
\end{align*}
To compute the integral $\int_{\Gm_{n}} \vs_{n}(\lm)\,\dlm$ note that if $\gm_{n}=0$ then $\vs_{n}(\lm) = (\tau_{n}-\lm)$ and hence $\int_{\Gm_{n}} \vs_{n}(\lm)\,\dlm = 0$.
On the other hand, if $\gm_{n}\neq 0$, then we may use~\eqref{s-root-sides} to compute
\[
  \int_{\Gm_{n}} \vs_{n}(\lm)\,\dlm
   = \ii \frac{\gm_{n}^{2}}{2} \int_{-1}^{1} \sqrt[+]{1-t^{2}}\,\dt
   = \ii\pi \frac{\gm_{n}^{2}}{4}.
\]
Consequently,
\begin{align*}
  \Om_{nn}^{(2)}
   = \frac{\gm_{n}^{2}}{4}\p*{\pi + \ell_{n}^{p/2}+\ell_{n}^{1+}}.\qed
\end{align*}

\end{proof}

\begin{thm}
\label{mkdv-freq}
\begin{equivenum}
\item
For any $n\in\Z$, the sum $-12\sum_{k\in\Z} k\Om_{nk}^{(2)}$ converges absolutely and locally uniformly on $\Ws^{p}$, $2 \le p < \infty$, to an analytic function which is an analytic extension of $\om_{n}^{(4)\star}$ given by~\eqref{omn-4-star}.

%
\item
For any $2 < p < \infty$, $\om^{(4)\star} = (\om_{n}^{(4)\star})_{n\in\Z} \colon \Ws^{p}\to \ell_{\C}^{-1,p/2}$ is a real analytic map which satisfies
\[
  \om_{n}^{(4)\star} + (12n\pi)I_{n} = n(\ell_{n}^{p/3} + \ell_{n}^{1+})
\]
locally uniformly on $\Ws^{p}$.
For $p=2$, $\om^{(4)\star}\colon \Ws^{2} \to \ell_{\C}^{-1,r}$ is real-analytic for any $r>1$.~\fish
\end{equivenum}
\end{thm}

\begin{rem}
Let $\Vs^{p/2} = \setdef{I = (I_{n}(\ph))_{n\in\Z}}{\ph\in \Wp^{p}}$ for $2 \le p < \infty$. Then $\Vs^{p/2}$ is an open and dense subset of $\ell^{p/2}$ which contains $\ell^{1}$.
Arguing as in the proof of \cite[Theorem~20.3]{Grebert:2014iq}, one sees that for any $2 \le p < \infty$, the frequency $\om_{n}^{(4)\star}$ is a real analytic function of the actions on $\Vs^{p/2}$.\map
\end{rem}

\begin{proof}
(i): In view of Lemma~\ref{decay-Omnk2} for $k\neq n$
\begin{equation}
  \label{Om-nk-2-est}
    k\Om_{nk}^{(2)}
  = (k-n)\Om_{nk}^{(2)} + n\Om_{nk}^{(2)}
  = \gm_{k}^{3}\ell_{k}^{1+}[n] + \frac{n}{n-k}\gm_{k}^{3}\ell_{k}^{1+}[n]
\end{equation}
locally uniformly on $\Ws^{p}$. In particular, the sum $-12\sum_{k\in\Z} k\Om_{nk}^{(2)}$ is absolutely and locally uniformly convergent to an analytic function on $\Ws^{p}$ for any $2 \le p < \infty$. Moreover, the identity $\om_{n}^{(4)\star} = -12\sum_{k\in\Z} k\Om_{nk}^{(2)}$, $n\in\Z$, holds for any real valued finite gap potential by Lemma~\ref{om-Om-4}. Consequently, $-12\sum_{k\in\Z} k\Om_{nk}^{(2)}$ is the unique analytic extension of $\om_{n}^{(4)\star}$ from the set of finite gap potentials to $\Ws^{p}$, $2 \le p < \infty$.

(ii):
By~\eqref{Om-nk-2-est} and Lemma~\ref{H-trans} we have
\[
  \sum_{k\neq n} k\Om_{nk}^{(2)}
  = \sum_{k\neq n} \gm_{k}^{3}\ell_{k}^{1+}[n] + n\sum_{k\neq n} \frac{\gm_{k}^{3}\ell_{k}^{1+}[n]}{n-k}
  = \ell_{n}^{\infty} + n(\ell_{n}^{p/3}+\ell_{n}^{1+}).
\]
Moreover, $n\Om_{nn}^{(2)} = n\frac{\gm_{n}^{2}}{4}(\pi + \ell_{n}^{p/2} + \ell_{n}^{1+})$ by Lemma~\ref{decay-Omnk2}, while
\[
  \frac{4I_{n}}{\gm_{n}^{2}} = 1 + \ell_{n}^{p/2} + \ell_{n}^{1+},
\]
by~\cite[Theorem~11.2]{Molnar:2016uq} and hence
\[
  \frac{\gm_{n}^{2}}{4} = I_{n} + \ell_{n}^{p/4} + \ell_{n}^{1},
\]
so that
\[
  \om_{n}^{(4)\star} = -12n\pi I_{n} + n(\ell_{n}^{p/3} + \ell_{n}^{1+})
\]
locally uniformly on $\Ws^{p}$. In particular, it follows that
\[
  \om_{n}^{(4)\star} = n (\ell_{n}^{p/2}+\ell_{n}^{1+}).\qed
\]
\end{proof}

\begin{proof}[Proof of Theorem~\ref{thm:nls-freq}.]
The statement of the theorem is a direct consequence of Theorem~\ref{mkdv-freq}.\qed
\end{proof}

\section{Symmetries of the mKdV frequencies}

The phase space $H^{3/2}(\T,\R)$ of the mKdV and \rmKdV equation corresponds to the subspace $\Ec_{r}^{2}\cap H_{r}^{3/2}$ of $H_{r}^{3/2}$ when they are viewed as equations in the NLS hierarchy. Here, for $2 \le p < \infty$
\[
  \Ec_{r}^{p} = \setd{\ph = (u,u)\in \FL_{r}^{p}}.
\]
As a first step towards the proof of Theorem~\ref{thm:wp-defocusing-mkdv}, we describe this subspace in Birkhoff coordinates. It was shown in \cite{Grebert:2002wb} that the space $\Ec_{r}^{p}$ in the case $p=2$ is characterized by the elements $\ph\in \FL_{r}^{2}$ satisfying
\[
  I_{-k}(\ph) = I_{k}(\ph)\forall k\in\Z,\quad \th_{-k}(\ph) = -\th_{k}(\ph)\forall k\in\Z\text{ with }I_{k}(\ph)\neq 0.
\]
The same characterization holds for $2 < p < \infty$, since $\Ec_{r}^{2}$ is dense in $\Ec_{r}^{p}$ and the Birkhoff map extends real analytically to $\Wp^{p}$. As a consequence, $z_{n}^{-} = \sqrt{I_{n}}\e^{-\ii \th_{n}}$ and $z_{n}^{+} = \sqrt{I_{n}}\e^{\ii \th_{n}}$ defined on $\FL_{r}^{p}\setminus Z_{n}$ satisfy for any $n\in\Z$,
\begin{equation}
  \label{zn-E-symmetry}
  z_{-n}^{-}(\ph) = z_{n}^{+}(\ph),\qquad
  z_{n}^{+}(\ph) = \ob{z_{n}^{-}(\ph)},\qquad
  \forall \ph\in \Ec_{r}^{p}.
\end{equation}
In particular, the Hamiltonian
\[
  \Hm_{2} = \sum_{m\in\Z} (2m\pi)I_{m}
\]
vanishes when restricted to $\Ec_{r}^{2} \cap H_{r}^{1/2}$.

Theorem~\ref{mkdv-freq} then can be used to prove the following

\begin{cor}
\label{cor:mkdv-freq}

\begin{equivenum}
\item
On $\Ec_{r}^{2}$, the frequencies $\om_{n}^{(4)}$, $n\in\Z$, are well defined and real analytic. For any $n\in\Z$, the restriction $\om_{n}^{(4)}\big|_{\Ec_{r}^{2}}$ is the $n$th mKdV frequency which we denote by $\om_{n}$.

\item On $\Ec_{r}^{2}$,
\[
  \om_{-n} = -\om_{n},\quad
  \om_{-n}^{(4)\star} = -\om_{n}^{(4)\star},\quad n\in\Z,
\]
and hence on $\Ec_{r}^{p}$, $2 \le p < \infty$,
\[
  \om_{-n}^{(4)\star} = -\om_{n}^{(4)\star},\qquad n\in\Z.\fish
\]

\end{equivenum}
\end{cor}

\begin{proof}
(i) At any finite gap potential of real type we have by~\eqref{om-star-Om-nk-4}
\[
  \om_{n}^{(4)} = (2n\pi)^{3} + 6\Hm_{2} + 12n\pi\Hm_{1} + \om_{n}^{(4)\star},
\]
and the right hand side is real analytic on $H_{r}^{1/2}$. Moreover, on $\Ec_{r}^{2}\cap H_{r}^{1/2}$ the Hamiltonian $\Hm_{2}$ vanishes identically. Since $\Hm_{1}$ is well defined and real analytic on $\FL_{r}^{2}$ (cf.~\cite{Grebert:2014iq}) and by Theorem~\ref{mkdv-freq}, $\om_{n}^{(4)\star}$, $n\in\Z$, are also well defined and real analytic on $\FL_{r}^{2}$, $\om_{n}^{(4)}$ is well defined on $\Ec_{r}^{2}$ and real analytic there.

(ii)
Let $\phi_{t}$ be a local flow for the vector field $X_{\Hm_{4}}$. The mKdV flow leaves the space $\Ec_{r}^{2}\cap H_{r}^{3/2}$ invariant, that is, $\phi_{0}\in \Ec_{r}^{2}\cap H_{r}^{3/2}$ implies that $\phi_{t}\in \Ec_{r}^{2}\cap H_{r}^{3/2}$ for all $t\in\R$. Consequently, we compute on $\Ec_{r}^{2}\cap H_{c}^{3/2}\setminus Z_{n}$
\[
  \om_{-n}^{(4)} = \ddt\th_{-n}\circ\phi_{t}\bigg|_{t=0} =
  -\ddt\th_{n}\circ\phi_{t}\bigg|_{t=0} = -\om_{n}^{(4)},\qquad n\in\Z,
\]
and the identity extends to all of $\Ec_{r}^{2}$ by continuity.
Finally, on $\Ec_{r}^{2}$
\[
  \om_{n}^{(4)\star} = \om_{n}^{(4)} - (2n\pi)^{3} - 12n\pi\Hm_{1},
\]
and the right hand side is an odd function of $n$. The claimed identity on $\Ec_{r}^{p}$ with $2\le p < \infty$ then follows by a density argument.\qed
\end{proof}

\section{Proof of Theorem~\ref{thm:wp-defocusing-mkdv}}
\label{s:wp}

In a first step we establish the results corresponding to the ones of Theorem~\ref{thm:wp-defocusing-mkdv} in Birkhoff coordinates and to this end introduce some more notation. By~\eqref{zn-E-symmetry}, the Birkhoff coordinates have the following symmetries on $\Ec_{r}^{p}$ 
\[
  z_{-n}^{-} = z_{n}^{+}\quad\text{and}\quad z_{n}^{-} = \ob{z_{n}^{+}}\quad \forall n\in\Z.
\]
For any $2 \le p < \infty$ we therefore define the real subspaces
\[
  e_{r}^{p} \defl \setdef{(z_{-},z_{+})\in\ell_{r}^{p}}{z_{-n}^{-} = z_{n}^{+}\forall n\in\Z},
\]
and note that $\Phi_{p}(\Ec_{r}^{p})$ is open and dense in $e_{r}^{p}$ whereas for $p=2$, $\Phi_{2}(\Ec_{r}^{2}) = e_{r}^{2}$. On $\Phi_{2}(\Ec_{r}^{2})$, the mKdV equation takes the form
\begin{equation}
  \label{mkdv-sys-bhf}
  \partial_{t}z_{n}^{-} = -\ii \om_{n}z_{n}^{-},\qquad
  \partial_{t}z_{n}^{+} = \ii \om_{n}z_{n}^{+},\qquad \forall n\in\Z,
\end{equation}
where for any $n\in\Z$ and $\ph\in \Ec_{r}^{2}$
\[
  \om_{n}(\ph) = (2n\pi)^{3} + 12n\pi \Hm_{1}(\ph) + \om_{n}^{(4)\star}(\ph).
\]
Similarly, the \rmKdV equation reads
\begin{equation}
  \label{rmkdv-sys-bhf}
  \partial_{t}z_{n}^{-} = -\ii \om_{n}^{\#}z_{n}^{-},\qquad
  \partial_{t}z_{n}^{+} = \ii \om_{n}^{\#}z_{n}^{+},\qquad \forall n\in\Z,
\end{equation}
where for any $n\in\Z$ and $\ph\in \Ec_{r}^{2}$
\[
  \om_{n}^{\#}(\ph) = (2n\pi)^{3} + \om_{n}^{(4)\star}(\ph).
\]
By Theorem~\ref{mkdv-freq}, $\om_{n}$, $\om_{n}^{\#}$, $n\in\Z$, are well defined on $\Ec_{r}^{2}$ and real analytic there. The corresponding solution maps are denoted by $\Sc_{\Phi}(t,\ph)$ and $\Sc_{\Phi}^{\#}(t,\ph)$, respectively. In more detail, given any $\ph\in \Ec_{r}^{2}$ let $z = (z_{-},z_{+}) = \Phi_{2}(\ph)$ and $\om_{n} = \om_{n}(\ph)$, $\om_{n}^{\#} = \om_{n}^{\#}(\ph)$. Then for any $t\in\R$,
\[
  \Sc_{\Phi}(t,\ph) = (\e^{-\ii \om_{n}t} z_{n}^{-}, \e^{\ii \om_{n}t}z_{n}^{+})_{n\in\Z},\quad
  \Sc_{\Phi}^{\#}(t,\ph) = (\e^{-\ii \om_{n}^{\#}t} z_{n}^{-}, \e^{\ii \om_{n}^{\#}t}z_{n}^{+})_{n\in\Z}.
\]
By Corollary~\ref{cor:mkdv-freq}, on $\Ec_{r}^{2}$ one has for any $n\in\Z$ that $\om_{-n} = -\om_{n}$ hence by the definition of $\om_{n}^{\#}$, one also has
\begin{equation}
  \label{om-n-sharp-symmetry}
  \om_{-n}^{\#} = -\om_{n}^{\#}
\end{equation}
It implies that $\Sc_{\Phi}(\cdot,\ph)$ and $\Sc_{\Phi}^{\#}(\cdot,\ph)$ both leave the subspace $e_{r}^{2}$ invariant. By Theorem~\ref{mkdv-freq}, the frequencies $\om_{n}^{\#}$, $n\in\Z$, real analytically extend to $\Ec_{r}^{p}$ for any $2 \le p < \infty$. It allows to extend the solution operator $\Sc_{\Phi}^{\#}(\cdot,\ph)$ as well.

\begin{thm}
\label{thm:wp-kdv}
\begin{equivenum}
\item
For any $\ph\in\Ec_{r}^{p}$ with $2\le p < \infty$, the curve
\[
  \R\to \ell_{r}^{p},\qquad t \mapsto \Sc_{\Phi}^{\#}(t,\ph) = (\e^{-\ii \om_{n}^{\#}t} z_{n}^{-}, \e^{\ii \om_{n}^{\#}t}z_{n}^{+})_{n\in\Z}
\]
is continuous and takes values in $e_{r}^{p}$.

\item
For any $2\le p < \infty$ and $T > 0$,
\[
  \Sc_{\Phi}^{\#}\colon \Ec_{r}^{p} \to C([-T,T],e_{r}^{p})
\]
is continuous.

\item
For any $2 < p < \infty$, $n\in\Z$, and $t > 0$, the coordinate functions $\ph\mapsto \e^{-\ii \om_{n}t} z_{n}^{-}$ and $\ph\mapsto \e^{\ii \om_{n}t} z_{n}^{+}$ cannot be extended continuously to elements $\ph \in \Ec_{r}^{p}\setminus \Ec_{r}^{2}$ with $z_{n}^{+}\neq 0$ and $z_{n}^{-}\neq 0$, respectively.~\fish
\end{equivenum}
\end{thm}

\begin{proof}
Item (i) and (ii) follow from \cite[Theorem~E.1]{Kappeler:2016uj}, taking into account that by Corollary~\ref{cor:mkdv-freq}, the identity~\eqref{om-n-sharp-symmetry} holds on $\Ec_{r}^{p}$ for any $2\le p < \infty$, implying that according to the above arguments, $\Sc_{\Phi}^{\#}(\cdot,\ph)$ leaves $e_{r}^{p}$ invariant for any $\ph\in\Ec_{r}^{p}$.

Item (iii) holds since on $\Ec_{r}^{p}\setminus\Ec_{r}^{2}$, $\Hc_{1}$ is infinite.\qed
\end{proof}

We are now in a position to prove Theorem~\ref{thm:wp-defocusing-mkdv}. First let us consider the case $p=2$. According to~\cite{Kappeler:2005gt}, for any initial datum $\ph\in\Ec_{r}^{2}$, there exists a unique global in time solution $\psi(t,x) = \psi(t,x,\ph)$ of~\eqref{mkdv} with $t\mapsto \psi(t,\cdot,\ph)$ being a continuous curve in $\Ec_{r}^{2}$.
Furthermore, for any $t\in\R$, the flow map $\Sc(t,\cdot)\colon \Ec_{r}^{2}\to\Ec_{r}^{2}$ is continuous and for any $T > 0$ the solution map
\[
  \Sc\colon \Ec_{r}^{2}\to C([-T,T],\Ec_{r}^{2}),\qquad \ph\mapsto \psi(\cdot,\cdot,\ph)
\]
is continuous as well. Since the mKdV flow preserves the $L^{2}$-norm of any initial datum in $\Ec_{r}^{2}$, it follows that for any solution $u(t)$ of~\eqref{mkdv} with $u(0) = u_{0}\in L^{2}(\T,\R)$ 
the curve $t\mapsto v(\cdot, t) = u(\cdot + 6\n{u}_{L^{2}}^{2}t,t)$ 
 is a solution of the \rmKdV equation with the same initial datum.
Therefore, on $\Ec_{r}^{2}$, solutions of~\eqref{mkdv} can be easily converted into solutions of~\eqref{rmkdv} and vice versa, hence for any $T > 0$, the solution map $\Sc^{\#}$ of the \rmKdV equation
\[
  \Sc^{\#}\colon \Ec_{r}^{2}\to C([-T,T],\Ec_{r}^{2})
\]
is real analytic and >>equivalent<< to the solution map $\Sc$ of the mKdV equation. In contrast to $\Sc$, the map $\Sc^{\#}$ can be continuously extended to $\Ec_{r}^{p}$ for any $2 < p < \infty$ as follows:

\begin{cor}
\label{cor:wp-mkdv}
\begin{equivenum}
\item
The \rmKdV equation is locally in time $C^{0}$-wellposed in $\FL^{p}(\T,\R)$ for any $2 < p < \infty$.

\item
The \rmKdV equation is globally in time $C^{0}$-wellposed in a neighborhood of $0$ in $\FL^{p}(\T,\R)$ for any $2 < p < \infty$.

\item
For any $2 < p < \infty$ and any $t > 0$, the solution map $\Sc$ of the mKdV equation does not extend continuously to any element of $\FL^{p}(\T,\R)\setminus L^{2}(\T,\R)$.~\fish

\end{equivenum}
\end{cor}

\begin{proof}
(i) Let $\ph\in \Ec_{r}^{p} (\cong \FL^{p}(\T,\R))$ with $2\le p < \infty$. Since the range of $\Phi_{p}$ is open in $\ell_{r}^{p}$, by Theorem~\ref{thm:wp-kdv}, there exists $T > 0$ and a neighborhood $U$ of $\ph$ in $\Ec_{r}^{p}$ so that
\[
  \Sc_{\Phi}^{\#}([-T,T],U) \subset \Phi_{p}(\Ec_{r}^{p}).
\]
Then the map
\[
  \Sc^{\#}\colon U\to C([-T,T],\Ec_{r}^{p}),\qquad \psi \mapsto (t\mapsto \Phi_{p}^{-1}( \Sc_{\Phi}^{\#}(t,\psi)))
\]
is continuous.

(ii) Since the range of $\Phi_{p}$ contains an open neighborhood of the origin, by Theorem~\ref{thm:wp-kdv} we can choose a neighborhood $U_{0}$ of $0$ in $\Ec_{r}^{p}$ so that
\[
  \Sc_{\Phi}^{\#}(\R,U_{0}) \subset \Phi_{p}(\Ec_{r}^{p}).
\]
Then the map
\[
  \Sc^{\#}\colon U_{0}\to C([-T,T],\Ec_{r}^{p}),\qquad \psi \mapsto (t\mapsto \Phi_{p}^{-1}( \Sc_{\Phi}^{\#}(t,\psi)))
\]
is continuous for any $T > 0$.

(iii) The claimed statement directly follows from item (iv) of Theorem~\ref{thm:wp-kdv}.\qed
\end{proof}

\begin{proof}[Proof of Theorem~\ref{thm:wp-defocusing-mkdv}.]
The theorem is a direct consequence of Corollary~\ref{cor:wp-mkdv}.\qed
\end{proof}

\begin{rem}
Since Theorem~\ref{bhf} provides Birkhoff coordinates locally around zero also for the focusing mKdV equation -- cf. \cite{Kappeler:2009ik} -- by the same approach as for the defocusing case, one can show that the renormalized focusing mKdV equation is globally in time $C^{0}$-wellposed locally around zero. Therefore, the focusing mKdV equation is ill-posed in $\Ec_{r}^{p}\setminus \Ec_{r}^{2}$ for $2 < p < \infty$.\map
\end{rem}

%
%
%
%
%
%

\appendix

\section{Discrete Hilbert Transform on $\ell^{p}$}

In this appendix we recall some well known facts on the discrete Hilbert transform on $\ell^{p}(\Z,\C)$ -- see e.g. \cite{Zygmund:1957up}.

\begin{lem}
\label{H-trans}
For any $2 \le p < \infty$, the discrete Hilbert transform
\[
  H\colon \ell_{\C}^{p}\to \ell_{\C}^{p},\qquad (Hx)_{n} =  \sum_{m\neq n} \frac{x_{m}}{m-n},
\]
defines an isomorphism on $\ell_{\C}^{p}$.~\fish
\end{lem}

To simplify notation we define $\sg^{0} = (n\pi)_{n\in\Z}$. The following bound of a modified version of the Hilbert transform is used in to estimate certain infinite products.

\begin{lem}
\label{A-trans}
Suppose $\sg = \sg^{0} + \tilde \sg$ and $\rho = \sg^{0} + \tilde \rho$ are sequences of complex numbers with $\tilde\sg,\tilde\rho\in\ell_{\C}^{\infty}$ such that for some $C > 0$
\[
  \abs{\rho_{m}-\sg_{n}} \ge C^{-1}\abs{m-n},\qquad m\neq n.
\]
Then for any $2 \le p < \infty$,
\[
  (Ax)_{n} = \pi\sum_{m\neq n} \frac{x_{m}}{\rho_{m}-\sg_{n}}
\]
defines a bounded linear operator $A$ on $\ell_{\C}^{p}$ whose bound depends only on $\n{\tilde\sg}_{\infty}$, $\n{\tilde\rho}_{\infty}$, $C$, and $p$.~\fish
\end{lem}

\section{Infinite products}

In this appendix we provide several estimates of infinite products with $\ell^{p}$ coefficients. The proofs for the case $p=2$ can be found in \cite{Grebert:2014iq} and their generalization to the case $2 \le p < \infty$ is straightforward -- see also~\cite{Molnar:2016uq}.

Recall from~\eqref{s-root} that for $\ph\in\Ws^{p}$ the standard root is defined by
\[
  \vs_{n}(\lm,\ph) = (\tau_{n}-\lm)\sqrt[+]{1- \gm_{n}^{2}/4(\tau_{n}-\lm)^{2}},\qquad
  \lm\notin G_{n},
\]
and is analytic in both variables on $(\C\setminus \ob{U_{n}})\times V_{\ph}$, where $U_{n}$, $V_{\ph}$, and $\Ws^{p}$ have been introduced in Section~2 -- see~\ref{iso-1}-\ref{iso-3}.

\begin{lem}
For any $\ph$ in $\Ws^{p}$, $2 \le p < \infty$, and $n\in\Z$,
\[
  f_{n}(\lm) = \frac{1}{\pi_{n}}\prod_{m\neq n} \frac{\vs_{m}(\lm)}{\pi_{m}}
\]
defines a function which is analytic in $\lm$ on $\C\setminus\bigcup_{m\neq n} G_{m}$ and analytic in both variables on $(\C\setminus \bigcup_{m\neq n} \ob{U_{m}})\times V_{\ph}$. Moreover, $f_{n}$ does not vanish on these domains.~\fish
\end{lem}

We denote $\sg^{0} = (n\pi)_{n\in\Z}$.

\begin{cor}
\label{wm-ana-quot}
For any $\ph\in \Ws^{p}$, $\tilde\sg =\sg-\sg^{0}\in\ell^{p}$, and $n\in\Z$, the function
\[
  \phi_{n}(\lm,\tilde\sg) = \prod_{m\neq n} \frac{\sg_{m}-\lm}{\vs_{m}(\lm)}
\]
is analytic on $(\C\setminus \bigcup_{m\neq n} \ob{U_{m}})\times\ell^{p}\times V_{\ph}$.~\fish
\end{cor}

We also introduce $\tau = (\tau_{n})_{n\in\Z}$ then $\tau-\sg^{0}\in \ell^{p}$ for any $\ph\in\Ws^{p}$.

\begin{lem}
\label{ana-quot-w-lp}
For any $\ph\in \Ws^{p}$, $2 \le p <\infty$, and $\sg-\sg^{0}\in\ell^{p}$ with $\sg-\tau\in\ell^{q}$, $1\le q\le p$,
\[
  \sup_{\lm\in U_{n}} \abs*{\prod_{m\neq n}\frac{\sg_{m}-\lm}{\vs_{m}(\lm)}-1}
   = \ell_{n}^{q} +\ell_{n}^{p/2} + \ell_{n}^{1+},
\]
locally uniformly on $\ell^{p}\times V_{\ph}$. As a consequence,
\[
  \frac{\sin \lm}{\lm-n\pi}\p*{\frac{1}{\pi_{n}}\prod_{m\neq n} \frac{\vs_{m}(\lm)}{\pi_{m}} }^{-1} = 1 + \ell_{n}^{q}+\ell_{n}^{p/2} + \ell_{n}^{1+},\qquad \lm\in U_{n}
\]
locally uniformly on $V_{\ph}$.~\fish
\end{lem}

%

\end{document}